\documentclass[11pt, twoside]{article}
\usepackage{amsmath,amsthm,amssymb}
\usepackage{times}
\usepackage{enumerate}
\def\titlerunning#1{\gdef\titrun{#1}}
\makeatletter
\def\author#1{\gdef\autrun{\def\and{\unskip, }#1}\gdef\@author{#1}}
\def\address#1{{\def\and{\\\hspace*{18pt}}\renewcommand{\thefootnote}{}%
\footnote {#1}}%
\markboth{\autrun}{\titrun}} \makeatother
\def\email#1{e-mail: #1}
\def\subjclass#1{{\renewcommand{\thefootnote}{}%
\footnote{\emph{Mathematics Subject Classification (2010):} #1}}}
\def\keywords#1{\par\medskip
\noindent\textbf{Keywords.} #1}

\usepackage{amssymb,latexsym}
\usepackage{mathrsfs}
\usepackage{amssymb}
\usepackage{amsmath,latexsym}
\usepackage{amscd,latexsym} 
\usepackage[all]{xy}

\def\qed{\hfill$\square$\smallskip}

\newtheorem{theorem}{Theorem}[section]

\newtheorem{corollary}[theorem]{Corollary}

\newtheorem{definition}[theorem]{Definition}

\newtheorem{remark}[theorem]{Remark}

\newtheorem{lemma}[theorem]{Lemma}

\newtheorem{proposition}[theorem]{Proposition}

\numberwithin{equation}{section}

\numberwithin{equation}{section}

\begin{document}
\baselineskip=17pt

\titlerunning{***}

\title{Existence results for coupled Dirac systems via Rabinowitz-Floer theory
}

\author{Wenmin Gong and Guangcun Lu}

\date{June 2, 2016}

\maketitle

\address{F1. Gong: School of Mathematical Sciences, Beijing Normal University,
    Beijing 100875, The People's Republic
 of China; \email{20133110010@mail.bnu.edu.cn}}

\address{F2. Lu: School of Mathematical Sciences, Beijing Normal University,
Laboratory of Mathematics  and Complex Systems,  Ministry of
  Education,    Beijing 100875, The People's Republic
 of China; \email{gclu@bnu.edu.cn}; Partially supported by the NNSF (grant no. 10971014 and 11271044 ) of China.}


\begin{abstract}
In this paper, we construct the Rabinowitz-Floer homology for the coupled Dirac system
\begin{equation*}
\left\{ \begin{aligned}
Du=\frac{\partial H}{\partial v}(x,u,v)\hspace{4mm} {\rm on}
\hspace{2mm}M,\\
Dv=\frac{\partial H}{\partial u}(x,u,v)\hspace{4mm} {\rm on}
\hspace{2mm}M,
\end{aligned} \right.
\end{equation*}
where $M$ is an $n$-dimensional compact Riemannian spin manifold, $D$ is the Dirac operator on $M$, and $H:\Sigma
M\oplus \Sigma M\to \mathbb{R}$ is a real valued superquadratic function of class $C^1$ with subcritical growth rates. Solutions of this system can be obtained from the critical points of a Rabinowitz-Floer functional on a product space of suitable fractional Sobolev spaces. In particular, we consider the $S^1$-equivariant $H$ that includes a nonlinearity of the form
$$
H(x,u,v)=f(x)\frac{|u|^{p+1}}{p+1}+g(x)\frac{|v|^{q+1}}{q+1},
$$
where $f(x)$ and $g(x)$ are strictly positive continuous functions on $M$, and  $p>1,q>1$ satisfy
$$
\frac{1}{p+1}+\frac{1}{q+1}>\frac{n-1}{n}.
$$
We establish the existence of a nontrivial solution by computing the Rabinowitz-Floer homology in the Morse-Bott situation.

\end{abstract}

\section{Introduction and main results}\label{sec:1}
\setcounter{equation}{0}

Let $(M,g)$ be an $n$-dimensional compact
oriented Riemannian manifold equipped with a spin structure
$\rho:P_{{\rm Spin}(M)}\rightarrow P_{{\rm SO}(M)}$, and let $\Sigma
M=\Sigma(M,g)=P_{{\rm Spin}(M)} \times_{\sigma}\Sigma_n$ denote the
complex spinor bundle on $M$. The latter is a complex vector bundle of
rank $2^{[n/2]}$ endowed with the spinorial Levi-Civita connection
$\nabla$ and a pointwise Hermitian scalar product.
In the following, let $\langle\cdot,\cdot\rangle$ always denote the real part of the
Hermitian product on $\Sigma M$. It induces a natural
 inner product $(u,v)_{L^2}=\int_M\langle u(x), v(x)\rangle dx$
 on the space $C^\infty(M,\Sigma M)$ of all $C^\infty$-sections of the bundle $\Sigma M$,
 where $dx$ is the Riemannian measure of $g$.
Denote by $L^2(M,\Sigma M)$ the completion Hilbert space of $C^\infty(M,\Sigma M)$.
The Dirac operator is an elliptic
differential operator of order one,  $D=D_g:C^{\infty}(M,\Sigma
M)\rightarrow C^{\infty}(M,\Sigma M)$, locally given by
$D\psi=\sum^n_{i=1}e_i\cdot\nabla_{e_i}\psi$ for $\psi\in
C^{\infty}(M,\Sigma M)$ and a local $g$-orthonormal frame
$\{e_i\}^n_{i=1}$ of the tangent bundle $TM$.
Consider Whitney direct sum  $\Sigma M\oplus\Sigma M$ of  $\Sigma M$ and itself,
and write a point of it as $(x,\xi,\zeta)$, where $x\in M$ and $\xi,\zeta\in \Sigma_x M$.  
Nolinear Dirac equations arise in many interesting problems in geometry and physics including Dirac-harmonic maps describing the generalized Weierstrass representation of surfaces in three-manifolds \cite{Fri} and the supersymmetric nonlinear sigma model in quantum field theory \cite{CJL,CJW1,CJW2}. In this paper we will
construct the Rabinowitz-Floer homology to study the following
system of the coupled semilinear Dirac equations:
\begin{equation} \label{eq:1.1}
\left\{ \begin{aligned}
Du=\frac{\partial H}{\partial v}(x,u,v)\hspace{4mm} {\rm on}
\hspace{2mm}M,\\
Dv=\frac{\partial H}{\partial u}(x,u,v)\hspace{4mm} {\rm on}
\hspace{2mm}M,
\end{aligned} \right.
\end{equation}
where $u,v\in C^1(M,\Sigma M)$ are spinors and $H:\Sigma M\oplus\Sigma M\to \mathbb{R}$ is a continuous  function. (\ref{eq:1.1})  is the Euler-Lagrange equation of
the functional
\begin{equation} \label{eq:1.2}
\mathfrak{L}_H(u,v)=\int_M\big(\langle Du,v\rangle-H(x,u,v)\big)dx.
\end{equation}
 The functional $\mathfrak{L}_H$ is strongly indefinite since the spectrum of the
operator $D$ is unbounded from below and above.

The problem (\ref{eq:1.1}) can be viewed as a spinorial analogue of other strongly indefinite variational problems such as infinite dynamical systems \cite{BaD1,BaD2} and elliptic systems \cite{AnV,Fed}, and in quantum physics it describes a coupled fermionic fields, and this is our main motivation for its study. A typical way to deal with such problems is the min-max method of Benci and Rabinowitz \cite{BeR}, including the mountain pass theorem, linking arguments and so on. For example, Isobe \cite{Iso1,Iso2} and authors \cite{GoL}
used this method to study the existence of solutions of generalized nonlinear Dirac equations
$Du=H_u(x, u)$ on a compact oriented spin Riemannian manifold.
Another way to solve them is homological approach
by using Morse theory or Floer homology as in \cite{Abb,BaL,Iso3, KrS}.
Inspired by the Rabinowitz-Floer homological method in \cite{AbS,AlF,CiF,CFA}, Maalaoui \cite{Maa}
studied the existence of solutions of the following subcritical Dirac equation
\begin{equation} \label{eq:1.3}
Du=|u|^{p-1}u \hspace{4mm}{\rm on}
\hspace{2mm}M,
\end{equation}
where $1<p<\frac{n+1}{n-1}$, by constructing Rabinowitz-Floer homology.
Recently, he also extended his results to a class of non-linear problems
with the so-called starshaped potential \cite{MaM}.
 Comparing these two methods, it seems that the homological approach is more ``intrinsic" in the sense that the topology of the space of solutions is invariant under perturbations of the subcritical exponent $p$.

In the following we assume that two real numbers $p,q>1$ satisfy
\begin{equation}\label{eq:1.4}
\frac{1}{p+1}+\frac{1}{q+1}>\frac{n-1}{n}.
\end{equation}
It is not hard to verify that we can choose a real number $s\in (0,1)$ such that
\begin{equation}\label{eq:1.4.1}
p<\frac{n+2s}{n-2s}\quad\hbox{and}\quad q<\frac{n+2-2s}{n+2s-2}.
\end{equation}
On nonlinearity $H$, we make the following hypotheses:

(\textbf{H0}) $H\in C^0(\Sigma M\oplus\Sigma M, \mathbb{R})$ is
 $C^1$ in the fiber direction, and $C^2$ in the fiber direction of
 $\Sigma M\oplus\Sigma M\setminus\{0\}$.

(\textbf{H1}) There exists a constant $c_0\in (0, 2)$  such that
\begin{equation} \label{eq:1.5}
\langle H_u(x,u,v),u\rangle+\langle H_v(x,u,v),v\rangle \geq2 H(x,u,v)- c_0
\end{equation}
for all $(x,u,v)$.

(\textbf{H2}) There exists a constant $c_1>0$ such that
\begin{gather}
\big|H_u(x,u,v)\big|\leq c_1\left(1+|u|^p+|v|^{\frac{p(q+1)}{p+1}}\right),\label{eq:1.6}\\
\big|H_v(x,u,v)\big|\leq c_1\left(1+|u|^{\frac{q(p+1)}{q+1}}+|v|^q\right). \label{eq:1.7}
\end{gather}

(\textbf{H3}) There exist constants $\delta>0$ and $c_2>0$ such that
for $|z|>\delta$ with $z=(u,v)$,
\begin{eqnarray} \label{eq:1.8}
&&\big|H_{uu}(x,u,v)\big|\leq c_2(1+|u|^{p-1}), \label{H3:1}\\ &&\big|H_{vv}(x,u,v)\big|\leq c_2(1+|v|^{q-1}),\label{H3:2}\\&&\big|H_{uv}(x,u,v)\big|\leq c_2,\quad\big|H_{vu}(x,u,v)\big|\leq c_2\label{H3:3}.
\end{eqnarray}

(\textbf{H4}) For any $a\in \mathbb{R}$, the map
$$\mathcal{T}:E_s\to L^{\frac{2n}{n+2s}}(M,\Sigma M)\times L^{\frac{2n}{n+2(1-s)}}(M,\Sigma M)$$
given by $\mathcal{T}(z)=\big(H_u(x,z),H_v(x,z)\big)^T$,
is bounded on the set
$$
\Sigma_a(H)=\Big\{z\in E_s\;\Big|\;\int_M H(x,z(x))dx\leq a\Bigr\},
$$
where $E_s=H^s(M,\Sigma M)\times H^{1-s}(M,\Sigma M)$,
see (\ref{eq:2.3}) for its definition.

{\it Note}: Since the equation (\ref{eq:1.1})
and the following assumptions (\textbf{H}2)-(\textbf{H}4) are invariant after
adding a constant to $H$, the assumption that $c_0<2$ in (\textbf{H}1) is unnecessary.
We assume it  so that the proof of Proposition~\ref{prop:3.1} becomes simple.

Consider the following typical examples  satisfying the above $({\bf H}0)-({\bf H}4)$,
\begin{equation} \label{eq:1.9}
H(x,u,v)=f(x)\frac{|u|^{p+1}}{p+1}+g(x)\frac{|v|^{q+1}}{q+1},
\end{equation}
where $f(x)$ and $g(x)$ are strictly positive continuous functions on $M$. Then~(\ref{eq:1.1}) reduces to the following form
\begin{equation} \label{eq:1.10}
\left\{ \begin{aligned}
Du=g(x)|v|^{p-1}v\hspace{4mm} {\rm on}
\hspace{2mm}M,\\
Dv=f(x)|u|^{q-1}u\hspace{4mm} {\rm on}
\hspace{2mm}M.
\end{aligned} \right.
\end{equation}
Note that $\int_MH(x,u,v)dx$ is not well-defined on the Hilbert space $H^{\frac{1}{2}}(M,\Sigma M)\times H^{\frac{1}{2}}(M,\Sigma M)$
unless we make a stronger hypothesis on the exponents $p,q$ as in \cite{Fed}.
The analytic framework in \cite{MaM} did not work well for our problem and so
Maalaoui and Martino's result cannot directly lead to the existence of the solutions of (\ref{eq:1.10}).
To overcome this difficulty, inspired by the ideas of
    Hulshof and Van der Vorst~\cite{HuV},
    we consider the following
    well-defined functional
   \begin{equation} \label{eq:1.11}
\mathcal{A}_H(u,v,\lambda)=\int_M\langle Du,v\rangle dx-\lambda\int_M\big(H(x,u,v)-1\big)dx
\end{equation}
 on a fractional Sobolev space $H^s(M,\Sigma M)\times H^{1-s}(M,\Sigma M)\times \mathbb{R}$ as
  an analogue of one  used by Rabinowitz ~\cite{Rab}.
From \S~\ref{sec:2} to \S~\ref{sec:6} we shall construct,  under the suitable assumptions on $H$,
 the  Rabinowitz-Floer homology in Morse and Morse-Bott situations, respectively.
For the latter case, that is,  the critical manifold consists of connected components with different dimensions,
in contrast to breaking the symmetry via a small perturbation to construct the $S^1$-equivalent homology as in \cite{Maa},
we shall follow \cite{Bou, BoO} construct the Morse-Bott homology as follows:  choose  a Morse function on the critical manifold
and define the chain complex to be the $\mathbb{Z}_2$-vector space generated by the critical points of this Morse function,
 while the boundary operator is defined by counting flow lines with cascades. The advantage of this method is that
there exists a nice grading for such a complex and the Rabinowitz-Floer homology for $H_0(x,u,v)=\frac{1}{2}(|u|^2+|v|^2)$
can be partially worked out
in Section~\ref{sec:7}. Based on these  we  prove the following result.

\begin{theorem}\label{th:1.1}
Assume that $n\geq 2$ and $0\notin {\rm Spec}(D)$.
Problem (\ref{eq:1.10}) has at least a nontrival solution $(u,v)\in C^1(M,\Sigma M)\times C^1(M,\Sigma M)$.
\end{theorem}
\noindent The same method can also be used to derive analogue existence results for a larger class of  homogeneous nonlinearities $H$.
Of course, if $H\in C^2(\Sigma M\oplus\Sigma M)$
satisfies $({\bf H}1)-({\bf H}4)$ then the functional $\mathfrak{L}_H$ in (\ref{eq:1.2}) is of class $C^2$ by Proposition~\ref{prop:2.1}.
The methods in \cite{Iso1} can be used to prove some results on existence and
multiplicity for solutions of (\ref{eq:1.1}) under certain further assumptions on $H$.
 We can use the saddle point reduction to study it as done in \cite{YJLu} for Dirac equations.
 These will be given in other places.\


{\bf Organization of the paper}. In section~2, we define a Rabinowitz-Floer functional on a suitable
product space of fractional Sobolev spaces, and the perturbed gradient flow. The aim of Section~3 is to prove the $(PS)_c$ condition and boundedness of the perturbed flows. In section~4, we define and study the relative index and moduli space of trajectories. Section~5 constructs the Rabinowitz Floer homology in Morse and Morse-Bott situations, and also proves continuation invariance of the homology.
In section~6, we establish the transversality result. Finally,  we compute the Rabinowitz-Floer homology and prove Theorem~\ref{th:1.1}
in section~7.


\section{The analytic framework}\label{sec:2}
\setcounter{equation}{0}
Let $(M, g)$ be as in Section~\ref{sec:1}.  The Dirac operator $D = D_g:C^\infty(M,\Sigma M)\to C^\infty(M,\Sigma M)$ is essentially self-adjoint in $L^2(M,\Sigma M)$ and its spectrum consists of an unbounded sequence of real numbers  (cf.~\cite{Fri,LaM}).
The well known Schr\"{o}dinger-Lichnerowicz formula implies that
 all eigenvalues of $D$ are nonzero  if $M$ has positive scalar curvature. Hereafter, {\it we assume}:
 $$
 0\notin {\rm spec}(D)\quad\hbox{and}\quad \int_M dx=1\quad\hbox{i.e., the volume of $(M,g)$  equals to $1$}.
 $$
  (The second  assumption  is only for simplicity, it is actually unnecessary for our result!).

Let $(\psi_k)_{k=1}^\infty$ be a complete $L^2$- orthonormal basis of eigenspinors corresponding to the eigenvalues $(\lambda_k)_{k=1}^\infty$ counted with multiplicity such that $|\lambda_k|\to \infty$ as $k\to \infty$.
For each $s\geq0$, let $H^s(M,\Sigma M)$ be the Sobolev space of fractional order $s$, its dual space is denoted by  $H^{-s}(M,$
$\Sigma M)$.
We have a  linear operator $|D|^s : H^s(M,\Sigma M)\subset L^2(M,\Sigma M) \to L^2(M,\Sigma M)$
 defined by
\begin{equation} \label{eq:2.1}
|D|^su =\sum\limits_{k=1}^{\infty}a_k|\lambda_k|^s\psi_k,
\end{equation}
where $u=\sum_{k=1}^{\infty}a_k\psi_k\in H^s(M,\Sigma M)$.
Since $0\notin {\rm spec}(D)$ the inverse $|D|^{-s}\in\mathscr{L}(L^2(M,\Sigma M))$
is  compact and self-adjoint. $|D|^s$  can be used to define a new inner product on $H^s(M,\Sigma M)$,
\begin{equation} \label{eq:2.2}
(u,v)_{s,2}:=(|D|^su,|D|^sv)_2.
\end{equation}
The induced norm $\|\cdot\|_{s,2}$
$=\sqrt{(\cdot,\cdot)_{s,2}}$ is equivalent to the usual one  on $H^s(M,\Sigma M)$ (cf.~\cite{Ada,Amm}).
For $r\in\mathbb{R}$ consider the Hilbert space
$$
\bar{\omega}^{2r}=\left\{{\bf a}=(a_1,a_2,\cdots)\;\Bigm|\;\sum^\infty_{k=1}a_k^2\lambda_k^{2r}<\infty\right\}
$$
with inner product
$$
\langle\!\langle {\bf a}, {\bf b}\rangle\!\rangle_{2r}=\sum^\infty_{k=1}\lambda_k^{2r}a_kb_k.
$$
Then $H^s(M,\Sigma M)$ can be identified with the Hilbert space $\bar{\omega}^{2s}$.
Hence
$$
H^{-s}(M,\Sigma M)=(H^{s}(M,\Sigma M))'
$$
can be identified with $\bar{\omega}^{-2s}$,
where the pairing between $\bar{\omega}^{-2s}$ and $\bar{\omega}^{2s}$ is given by
$$
\langle {\bf a}, {\bf b}\rangle=\sum^\infty_{k=1}a_kb_k.
$$
It follows that $|D|^{-2s}$ gives a Hilbert space isomorphism from
$H^{-s}(M,\Sigma M)$ to $H^{s}(M,\Sigma M)$ with respect to the equivalent new inner products
as in (\ref{eq:2.2}). Moreover we have a continuous inclusion $L^2(M,\Sigma M)\hookrightarrow H^{-s}(M,\Sigma M)$
and
\begin{equation} \label{eq:2.2.1}
(|D|^{-2s}u, v)_{s,2}:=(u,v)_2\quad\forall u,v\in L^2(M,\Sigma M).
\end{equation}

Consider the Hilbert space
\begin{equation} \label{eq:2.3}
E_s:=H^s(M,\Sigma M)\times H^{1-s}(M,\Sigma M)
\end{equation}
with norm $\|z\|:=(\|u\|_s^2+\|v\|_{1-s}^2)^{\frac{1}{2}}$ for $z=(u,v)\in E_s$. By the Sobolev embedding theorem, we have the compact embedding $E_s\hookrightarrow L^{p+1}(M,\Sigma M)\times L^{q+1}(M,\Sigma M)$.
 Let $E_s^*=H^{-s}(M,\Sigma M)\times H^{-(1-s)}(M,\Sigma M)$, which is the dual space of $E_s$.
Then
$$
\mathcal{D}_s:=\begin{pmatrix}
 |D|^{-2s} & 0 \\ 0 & |D|^{-2(1-s)}
\end{pmatrix}:E_s^*\to E_s
$$
is a Hilbert space isomorphism by the arguments above (\ref{eq:2.2.1}) and
\begin{equation}\label{eq:2.4}
(\mathcal{D}_sz_1,z_2)_{E_s}=(z_1,z_2)_{L^2}
\end{equation}
for any $z_1$, $z_2\in L^2(M,\Sigma M)\times L^2(M,\Sigma M)$.

Since $M$ is compact, by the assumption ({\bf H1}) we have constants $C_1, C_2>0$
such that
\begin{equation} \label{eq:2.4.1}
|H(x,u,v)|\geq C_1(|u|^2+|v|^2)-C_2\quad\forall (x,u,v),
\end{equation}
and by the assumption ({\bf H2}) we can use Young's inequality to derive
\begin{equation} \label{eq:2.4.2}
|H(x,u,v)|\leq C(1+|u|^{p+1}+|v|^{q+1})\quad\forall (x,u,v)
\end{equation}
for some constant $C>0$. (Later on, we also use  $C$ to denote various positive
constants independent of $u$ and $v$ without special statements).
  (\ref{eq:2.4.1}) and (\ref{eq:2.4.2}) show  that the nonlinearity $H$ is asymptotically quadric or superquadric.

From now on  we also assume  that
\begin{equation}\label{eq:2.4.3}
H\in C^2(\Sigma M\oplus\Sigma M).
\end{equation}

\begin{proposition}\label{prop:2.1}
Assume that $H\in C^1(\Sigma M\oplus\Sigma M)$
satisfies $({\bf H}1)-({\bf H}2)$ and $({\bf H}4)$. Then the functional  $\mathcal{H}:E_s\to \mathbb{R}$
 defined by
 \begin{equation} \label{eq:2.4.4}
\mathcal{H}(x,u,v)=\int_MH(x,u(x),v(x))dx,
\end{equation}
 is of class $C^1$,  its derivation at $(u,v)\in E_s$ is given by
\begin{equation}\label{eq:2.4.5}
\mathcal{H}^\prime(u,v)(\xi,\zeta)=\int_M\big(\langle H_u(x,u,v),\xi\rangle +\langle H_v(x,u,v),\zeta\rangle\big) dx\quad\forall (\xi,\zeta)\in E_s,
\end{equation}
and $\mathcal{H}^\prime:E_s\to E_s^\ast\equiv E_s$ is a compact map.
Furthermore, if this $H$ also belongs to  $C^2(\Sigma M\oplus\Sigma M)$
and satisfies $({\bf H}3)$, then
$\mathcal{H}$ is of class $C^2$.
\end{proposition}

 \begin{remark}\label{rmk:2.2}
{\rm  If the real numbers $p,q$ satisfy
$$
1<p,q<\min\bigg\{\frac{n+2s}{n-2s},
\frac{n+2(1-s)}{n-2(1-s)}\bigg\}
$$
for some $s\in(0,1)$, which implies (\ref{eq:1.4}), the above space $E_s$ can be replaced by $E_{\frac{1}{2}}$.
In particular, for $n=\dim M=2$ and $2<p,q<3$, we can prove that the
functional  $\mathcal{H}:E_{\frac{1}{2}}\to \mathbb{R}$ is of class $C^3$
provided that  $H\in C^3(\Sigma M\oplus\Sigma M)$ satisfies  $({\bf H}1)-({\bf H}4)$,
and that suitable growth conditions on $H_{uuu}$, $H_{vvv}$, $H_{uuv}$ and $H_{uvv}$ are applied.
Of course, for $n=\dim M=1$ it can also be proved that
the functional  $\mathcal{H}$ is of class $C^\infty$ on $E_{\frac{1}{2}}$ if
$H\in C^\infty(\Sigma M\oplus\Sigma M)$ satisfies suitable conditions.
}
\end{remark}
For the sake of completeness we shall give the proof of Proposition~\ref{prop:2.1} in Appendix A.

It follows that the Rabinowitz-Floer functional $\mathcal{A}_H$ in (\ref{eq:1.11})
 is  of class $C^2$ on Hilbert space
 $\mathcal{E}:=E_s\times \mathbb{R}$
 with inner product
\begin{equation}\label{eq:2.5}
\big((\xi_1,\mu_1),(\xi_2,\mu_2)\big)_{\mathcal{E}}=(\xi_1,\xi_2)_{E_s}+\mu_1\cdot\mu_2
\end{equation}
for $(\xi_i,\mu_i)\in \mathcal{E}$, $i=1, 2$. Moreover,
   $(u,v,\lambda)\in \mathcal{E}$ is a critical point of $\mathcal{A}_H$ if and only if
\begin{eqnarray} \label{eq:2.6}
\left\{ \begin{array}{l}
Du=\lambda H_v(x,u,v)\hspace{8mm} {\rm on}
\hspace{2mm}M,\\
Dv=\lambda H_u(x,u,v)\hspace{8mm} {\rm on}
\hspace{2mm}M,\\
\textstyle\int_M H(x,u,v)dx=1\hspace{6mm} {\rm on}\hspace{2mm}M.
\end{array} \right.
\end{eqnarray}
Since $\langle Du,v\rangle=\langle u,Dv\rangle$ and $\int_M\langle Du,v\rangle dx=(Du,v)_2$,
 the functional $\mathcal{A}_H$ can be written as
\begin{equation} \label{eq:2.7}
\mathcal{A}_H(z,\lambda)=\frac{1}{2}\int_M\langle Lz(x),z(x)\rangle dx-\lambda\int_M (H(x,z(x))-1)dx,
\end{equation}
where
\begin{equation}\notag
L=\begin{pmatrix}
 0 & D \\ D & 0
\end{pmatrix}.
\end{equation}
Note that $\int_M\langle Lz(x),z(x)\rangle dx=(Lz,z)_2=(\mathcal{D}_sLz,z)_{E_s}$ by (\ref{eq:2.4}).
We deduce that the gradient of $\mathcal{A}_H$ with respect to the metric (\ref{eq:2.5}) is given by
\begin{equation}\label{eq:2.8}
\nabla \mathcal{A}_H(z,\lambda)=\begin{pmatrix}
\mathcal{D}_s\{Lz-\lambda H_z(x,z)\} \\
 -\int_M\big(H(x,z)-1\big)dx
\end{pmatrix},
\end{equation}
where $H_z(x,z)=(H_u(x,u,v),H_v(x,u,v))^T$.  Proposition~\ref{prop:2.1}
implies that $\nabla\mathcal{A}_H$ is of class $C^1$ on $E_s$. Hence  the following system of PDE's
\begin{eqnarray} \label{eq:2.9}
\left\{ \begin{array}{l}
\frac{\partial z}{\partial t}=-\mathcal{D}_s\{L z-\lambda H_z(x,z)\},\\[4pt]
\frac{\partial \lambda}{\partial t}=\textstyle\int_M (H(x,z)-1)dx.
\end{array} \right.
\end{eqnarray}
has a local flow on $\mathcal{E}$. But the initial value problem for the $L^2$ - gradient flow is ill-posed since the spectrum of $D$ is unbounded from below. We work on $E_s$ which makes the absence of the symmetry of $u$ and $v$ by imposing more regularity of $u$ than of $v$ if $p$ is large and $q$ is small, and vice versa.

\noindent{\bf The perturbed flows.} To obtain transversality, We shall follow the idea of Angenent and Vorst \cite{AnV} to perturb the metric on $\mathcal{E}=E_s\oplus\mathbb{R}$ and thus make all connecting orbits between critical points to be transverse.

Let $C^2=C^2(M,\Sigma M\oplus\Sigma M)$, which is a separable Banach space; see \cite{KrP}.
By the definition (cf. \cite{HiP}), a nuclear  operator $T$ from
$\mathcal{E}$ to $C^2\oplus\mathbb{R}$ is a bounded linear operator which
can be written as an absolutely convergent sum
$\sum_{k=1}^\infty y_k\otimes x_k^*$, where $x_k^*\in \mathcal{E}^*$ and $y_k\in C^2\oplus\mathbb{R}$.
The norm of $T$ is defined by
\begin{equation*}
\|T\|_{\mathcal{NS}}={\rm inf}\sum\limits_{k=1}^\infty\|x^*_k\|_{\mathcal{E}^*}\|y_k\|
_{C^2\oplus\mathbb{R}}.
\end{equation*}
Consider the space
\begin{eqnarray}\notag
\mathcal{NS}(\mathcal{E},C^2\oplus\mathbb{R}):=\left\{K\in \mathscr{L}(\mathcal{E},C^2\oplus\mathbb{R})\bigg|\begin{array}{l}
K\hbox{ is nuclear and symmetric with}\\
\hbox{respect to the inner product of }\mathcal{E}
\end{array}\right\}.
\end{eqnarray}
It is a separable Banach space with respect to the above norm
$\|\cdot\|_{\mathcal{NS}}$, and contains the space of finite rank operator from
$\mathcal{E}$ to $C^2\oplus\mathbb{R}$ as a dense subspace.

 Let $K:\mathcal{E}\to\mathcal{NS}(\mathcal{E},C^2\oplus\mathbb{R})$ be a smooth map of form
\begin{equation}\label{met:1}
K(w)=e^{-\|w\|_{\mathcal{E}}^2}\tilde{K}(w),
\end{equation}
where $\tilde{K}\in C^\infty(\mathcal{E}, \mathcal{NS}(\mathcal{E},C^2\oplus\mathbb{R}))$  satisfies the Gevrey type estimates
\begin{equation}\label{met:2}
\sup_{n\ge 0}\frac{\sup\limits_{w\in \mathcal{E}}\|\tilde{K}^{(n)}(w)\|_{\mathscr{L}_n(\mathcal{E},\mathcal{NS}
(\mathcal{E},C^2\oplus\mathbb{R}))}}{(n!)^2}<\infty.
\end{equation}
Denote by $\mathbf{K}_0$ the set of such maps $K$ such that
\begin{equation}\label{met:3}
\sup_{w\in \mathcal{E}}\|K(w)\|_{\mathscr{L}(\mathcal{E}, \mathcal{E})}<\frac{1}{2}.
\end{equation}
(Note: for any $w\in \mathcal{E}$, $K(\omega)\in \mathcal{NS}(\mathcal{E},C^2\oplus\mathbb{R})
\subset\mathscr{L}(\mathcal{E},C^2\oplus\mathbb{R})$ and $C^2\oplus\mathbb{R}\hookrightarrow\mathcal{E}$
is continuous, so $K(\omega)\in\mathscr{L}(\mathcal{E}, \mathcal{E})$.)
The norm of $K$ is defined to be the left side of inequality in  (\ref{met:2}).
 Then $\mathbf{K}_0$ is a Banach space with respect to this norm. Note that the space  $\mathbf{K}_0$ contains maps of the form
\begin{equation}\label{met:4}
\rho(\|w-w_0\|)k_0,
\end{equation}
where $k_0\in \mathcal{NS}(\mathcal{E},C^2\oplus\mathbb{R})$ is a constant, and $\rho(t)=e^{-1/(1-t^2)}$ for $t<1$, and $\rho(t)=0$ for $t\geq 1$. We define  a  closed linear subspace of $\mathbf{K}_0$,
 \begin{equation}\label{met:4.1}
\mathbf{K}={\rm span}(\{\hbox{ all maps of form (\ref{met:4})}\}).
\end{equation}
Each $K\in \mathbf{K}$ can yield a perturbed Riemannian metric $g^K$ on $\mathcal{E}$ defined by
\begin{equation}\label{met:5}
g_w^K\big(\xi_1,\xi_2\big)=
(\xi_1,(I+K(w))^{-1}\xi_2)_\mathcal{E},
\end{equation}
where $\xi_i\in T_w\mathcal{E}=\mathcal{E}, i=1, 2$.
Then the gradient of $\mathcal{A}_H$ with respect to $g^K$ is given by
\begin{equation}\label{eq:2.10}
\nabla^K \mathcal{A}_H(w)=(I+K(w))\nabla \mathcal{A}_H(w),
\end{equation}
and the modified gradient flow becomes
\begin{equation} \label{eq:2.11}
\frac{dw(t)}{dt}+\nabla^K \mathcal{A}_H(w(t))=0.
\end{equation}
Denote by ${\rm Pr_1}$ the projection from $\mathcal{E}$ to $E_s$. From (\ref{met:1}) we get
\begin{equation} \label{eq:2.12}
\|{\rm Pr_1}(K(w)\nabla\mathcal{A}_H(w))\|_{C^2}\leq C.
\end{equation}


\begin{proposition}\label{prop:2.2}
For any $x\in\mathcal{E}\setminus\{0\}$ and $y\in C^2\oplus\mathbb{R}$, there exists a $K\in\mathcal{NS}(\mathcal{E},C^2\oplus\mathbb{R})$ satisfying $K(x)=y$.
\end{proposition}

\begin{proof}
As noted above the space of finite rank operators from $\mathcal{E}$ to $C^2\oplus\mathbb{R}$ is dense in the space of nuclear operators. If $(x,y)_\mathcal{E}=0$, by choosing $\xi\in C^2\oplus\mathbb{R}$ with $(x,\xi)_\mathcal{E}\neq0$ we define
\begin{equation} \label{eq:2.13}
K(w)=\frac{(w,\xi)_\mathcal{E}}{(x,\xi)_\mathcal{E}}y+\frac{(w,y)_\mathcal{E}}
{(x,\xi)_\mathcal{E}}\xi,\quad\forall w\in\mathcal{E}.
\end{equation}
If $(x,y)_\mathcal{E}\neq0$, we put
\begin{equation} \label{eq:2.14}
K(w)=\frac{(w,y)_\mathcal{E}}
{(x,y)_\mathcal{E}}y,\quad\forall w\in\mathcal{E}.
\end{equation}
In both cases, $K$ is a finite rank operator which is symmetric with respect to
 the inner product in (\ref{eq:2.5}).
\end{proof}

\section{$(PS)_c$ condition and boundedness of the perturbed flows}\label{sec:3}
\setcounter{equation}{0}

In this section we always assume that $H\in C^2(\Sigma M\oplus\Sigma M)$
satisfies $({\bf H}1)-({\bf H}4)$ without special statements.

\subsection{$(PS)_c$ condition}\label{sec:3.1}


\begin{proposition}\label{prop:3.1}
 Suppose that $H\in C^1(\Sigma M\oplus\Sigma M)$
satisfies $({\bf H}1)-({\bf H}2)$ and $({\bf H}4)$.
 Then the functional $\mathcal{A}_H$ satisfies the $(PS)_c$ condition;
that is, suppose that a sequence $\{(z_k,\lambda_k)\}^\infty_{k=1}\subset\mathcal{E}$
satisfies   $\mathcal{A}_H(z_k,\lambda_k)\to c\in\mathbb{R}$ and
$$
 \|\nabla\mathcal{A}_H(z_k,\lambda_k)\|_{\mathcal{E}}=\|d\mathcal{A}_H(z_k,\lambda_k)\|_{\mathcal{E}^*}\to0 \;\hbox{as $k\to \infty$},
 $$
 then it has a convergent subsequence.
\end{proposition}
\proof Since
  $\|\nabla\mathcal{A}_H(z_k,\lambda_k)\|^2_{\mathcal{E}}=
 \|\mathcal{D}_s\{L z_k-\lambda_k H_z(x,z_k)\}\|_{E_s}^2+ \left(\int_M(H(x,z_k)-1)dx\right)^{2}
 $  by (\ref{eq:2.8}),   we have
  \begin{eqnarray} \label{eq:3.1}
\left\{ \begin{array}{l}
 \epsilon_k:=\|\mathcal{D}_s\{L z_k-\lambda_k H_z(x,z_k)\}\|_{E_s}\to 0,\\
 \varepsilon_k:=\int_M(H(x,z_k)-1)dx\to 0
\end{array} \right.
\end{eqnarray}
as $k\to\infty$. Recalling $\int_Mdx=1$,
from the assumption $({\bf H}1)$ we derive
\begin{eqnarray} \label{eq:3.2}
\int_M\langle H_z(x,z_k),z_k\rangle dx\geq 2(1+\varepsilon_k)-c_0\ge 2-c_0>0\quad\forall k\in\mathbb{N}.
\end{eqnarray}
 Moreover,  the definition of $d\mathcal{A}_H$ implies
\begin{eqnarray} \label{eq:3.3}
&&\langle d\mathcal{A}(z_k,\lambda_k),(z_k,\lambda_k)\rangle\notag\\
&=&\big(\mathcal{D}_s\{L z_k-\lambda_k H_z(x,z_k)\},z_k\big)_{E_s}
-\lambda_k\int_M(H(x,z_k)-1)dx\notag\\
&=&(L z_k,z_k)_{L^2}-\lambda_k
(H_z(x,z_k),z_k)_{L^2}-\lambda_k\int_M (H(x,z_k)-1)dx\notag\\
&=&2\mathcal{A}_H(z_k,\lambda_k)-\lambda_k
(H_z(x,z_k),z_k)_{L^2}+\varepsilon_k\lambda_k\notag\\
&=&2c-\lambda_k
(H_z(x,z_k),z_k)_{L^2}+\varepsilon_k\lambda_k+ o(1).
\end{eqnarray}
Then it follows from (\ref{eq:3.2}) and (\ref{eq:3.3}) that for some constant $C>0$ and all $k\in\mathbb{N}$,
\begin{eqnarray} \label{eq:3.4}
|\lambda_k|&\leq&\frac{1}{2-c_0}|\lambda_k(H_z(x,z_k),z_k)_{L^2}|\notag\\
&\leq&\frac{1}{2-c_0}\{|\varepsilon_k\lambda_k|+2|\mathcal{A}_H
(z_k,\lambda_k)|+|d\mathcal{A}_H(z_k,\lambda_k),(z_k,\lambda_k)\rangle|\}
\notag\\
&\leq&\frac{1}{2-c_0}\{|\varepsilon_k\lambda_k|+2|\mathcal{A}_H
(z_k,\lambda_k)|+\|d\mathcal{A}_H(z_k,\lambda_k)\|_{\mathcal{E}^\ast}\cdot\|(z_k,\lambda_k)\|_{\mathcal{E}}\}
\notag\\
&\leq&C\{1+|\varepsilon_k||\lambda_k|+(\epsilon_k+\varepsilon_k)(\|z_k\|+|\lambda_k|)\}.
\end{eqnarray}

Next we estimate the $E_s$- norms of $z_k$. Since $\int_MH(x,z_k)dx=1+\varepsilon_k$ is bounded, by the Sobolev imbedding theorem and the assumption $({\bf H}4)$, we get
\begin{eqnarray} \label{eq:3.5}
\|\mathcal{D}_s H_z(x,z_k)\|_{E_s}&\leq&\||D|^{-2s}H_u(x,z_k)\|_{s,2}+\||D|^{-2(1-s)}
H_v(x,z_k)\|_{1-s,2}\notag\\
&=&\|H_u(x,z_k)\|_{-s,2}+\|H_v(x,z_k)\|_{-(1-s),2}\notag\\
&\leq&C\big(\|H_u(x,z_k)\|_{\frac{2n}{n+2s}}+\|H_v(x,z_k)\|
_{\frac{2n}{n+2(1-s)}}\big)\leq C,
\end{eqnarray}
where $C>0$ denotes different  constants.
Note that the composition operator $\mathcal{D}_s L:E_s\to E_s$ is an isometry.
Thus
\begin{eqnarray} \label{eq:3.7}
\|z_k\|^2_{E_s}&=&\|\mathcal{D}_s L z_k\|^2_{E_s}\notag\\
&=&\big|\big(\mathcal{D}_s L z_k,\mathcal{D}_s\{L z_k-\lambda_k H_z(x,z_k)\}\big)_{E_s}+\big(\mathcal{D}_s L z_k,
\lambda_k\mathcal{D}_s H_z(x,z_k)\big)_{E_s}\big|\notag\\
&\leq&\epsilon_k\|\mathcal{D}_s L z_k\|_{E_s}+|\lambda_k|
\|\mathcal{D}_s L z_k\|_{E_s}\|\mathcal{D}_s H_z(x,z_k)\|_{E_s}\notag\\
&\leq&\epsilon_k\|z_k\|_{E_s}+C|\lambda_k|
\|z_k\|_{E_s}.
\end{eqnarray}
Combining  (\ref{eq:3.4}) with (\ref{eq:3.7}), we deduce that
\begin{eqnarray} \label{eq:3.8}
|\lambda_k|+\|z_k\|_{E_s}\leq C\{1+|\varepsilon_k||\lambda_k|+\epsilon_k(\|z_k\|_{E_s}+|\lambda_k|)\},
\end{eqnarray}
which implies that both $z_k$ and $\lambda_k$ are bounded. Passing to a subsequence, we may assume that $z_k$ converges weakly in $E_s$ to $z=(u,v)$ and $\lambda_k$ converges to $\lambda\in\mathbb{R}$. Let $b_k$ be the first component of $\nabla \mathcal{A}_H(z_k,\lambda_k)$, i.e., $b_k=\mathcal{D}_sL z_k-\lambda_k\mathcal{D}_s H_z(x,z_k)$, which converges to zero. Since the operator $\mathcal{D}_s L:E_s\to E_s$ is an isometry,
and $\mathcal{H}^\prime$ is compact by Proposition~\ref{prop:2.1},
 we conclude that
\begin{eqnarray} \label{eq:3.9}
z_k=(\mathcal{D}_s L)^{-1}b_k-\lambda_k(\mathcal{D}_s L)^{-1}\mathcal{D}_s H_z(x,z_k)
\end{eqnarray}
converges in $E_s$ and so $\|z_k-z\|_{E_s}\to 0$. This shows that $\mathcal{A}_H$ satisfies the $(PS)_c$ condition.
\qed

Obverse that in the proof of boundedness of $\|(z_k,\lambda_k)\|_\mathcal{E}$ we only use the boundedness of $\mathcal{A}_H(z_k,\lambda_k)$ and the condition that $\|\nabla \mathcal{A}_H(z_k,\lambda_k)\|_\mathcal{E}$ is small enough. Consequently, we have

\begin{corollary}\label{coro:3.1}
Suppose $|\mathcal{A}_H(z,\lambda)|<R$ for some constant $R>0$. Then there exist $\epsilon>0$ and $C=C(R)$ such that each $(z,\lambda)\in \mathcal{E}$ with $\|\nabla \mathcal{A}_H(z,\lambda)\|_{\mathcal{E}}\leq\epsilon$ satisfies $\|(z,\lambda)\|_{\mathcal{E}}<C$.
\end{corollary}

\subsection{Boundedness in $\mathcal{E}$ for the autonomous flow}\label{sec:3.2}

Since $\mathcal{A}_H$ is of class $C^2$ under our assumptions, the local flow of $\nabla^K \mathcal{A}_H$
always exists.

\begin{proposition}\label{prop:3.2}
Assume that $\mathbb{R}\ni t\mapsto (z(t),\lambda(t))\in \mathcal{E}$ is a flow line of
$-\nabla^K \mathcal{A}_H$ between critical points and that
$\{\mathcal{A}_H(z(t),\lambda(t))\,|\,t\in \mathbb{R}\}\subset [a, b]$
for some two real numbers $a<b$.  Then there exists a constant $C_1=C_1(R,a,b)>0$ such that
$\|(z(t),\lambda(t))\|_{\mathcal{E}}\le C_1$ for all $t\in\mathbb{R}$.
\end{proposition}
\proof
Let $\epsilon$ be as in Corollary~\ref{coro:3.1} with $R=\max\{|a|,|b|\}$. For $s\in\mathbb{R}$,
if $\big\|\nabla \mathcal{A}_H\bigl(z(s),\lambda(s)\bigr)\big\|_{\mathcal{E}}<\epsilon$
we define $\tau(s)=0$; otherwise, since $\big\|\nabla \mathcal{A}_H\bigl(z(s+t),\lambda(s+t)\bigr)\big\|_{\mathcal{E}}\to 0$ as $t\to +\infty$ we have
$\big\{t\geq 0\;\big|\; \big\|\nabla \mathcal{A}_H\bigl(z(s+t),\lambda(s+t)\bigr)\big\|_{\mathcal{E}}<\epsilon\big\}
\ne\emptyset$ and hence
\begin{eqnarray}\label{eq:3.10a}
\tau(s):=\inf \left\{t\geq 0\;\Bigm|\; \big\|\nabla \mathcal{A}_H\bigl(z(s+t),\lambda(s+t)\bigr)\big\|_{\mathcal{E}}<\epsilon\right\}
\end{eqnarray}
is a nonnegative real number. Clearly, in the latter case it holds that
$$
\big\|\nabla \mathcal{A}_H\bigl(z(s+t),\lambda(s+t)\bigr)\big\|_{\mathcal{E}}\ge\epsilon\quad
\forall t\in [s, s+\tau(s)].
$$
Moreover, (\ref{met:3}) implies that
$((I+K)w,w)_{\mathcal{E}}\geq\frac{1}{2}\|w\|^2_{\mathcal{E}}$ for any $w\in\mathcal{E}$.
Then
\begin{eqnarray*}
\|\nabla^K \mathcal{A}_H(z(t),\lambda(t))\|^2_{g^K}&=&\big((I+ K((z(t),\lambda(t))) \nabla \mathcal{A}_H(z(t),\lambda(t)), \nabla \mathcal{A}_H(z(t),\lambda(t))\big)_{\mathcal{E}}\\
&\ge&\frac{1}{2}\big(\nabla \mathcal{A}_H(z(t),\lambda(t)), \nabla \mathcal{A}_H(z(t),\lambda(t))\big)_{\mathcal{E}},
\end{eqnarray*}
and thus
\begin{eqnarray}\label{eq:3.10}
b-a&\geq&\mathcal{A}_H(z(-\infty),\lambda(-\infty))
-\mathcal{A}_H(z(+\infty),\lambda(+\infty))\notag\\
&=&-\int^{+\infty}_{-\infty}\frac{d}{dt}\mathcal{A}_H(z(t),\lambda(t))dt\notag\\
&=&\int^{+\infty}_{-\infty}\|\nabla^K \mathcal{A}_H(z(t),\lambda(t))\|^2_{g^K}dt\notag\\
&\geq&\frac{1}{2}\int^{+\infty}_{-\infty}\|\nabla \mathcal{A}_H(z(t),\lambda(t))\|^2_{\mathcal{E}}dt\notag\\
&\geq&\frac{1}{2}\int^{s+\tau(s)}_s\|\nabla \mathcal{A}_H(z(t),\lambda(t))\|^2_{\mathcal{E}}dt\notag\\
&\geq&\frac{1}{2}\tau(s)\epsilon^2
\end{eqnarray}
if $\big\|\nabla \mathcal{A}_H\bigl(z(s),\lambda(s)\bigr)\big\|_{\mathcal{E}}\ge\epsilon$.
So for any $s\in\mathbb{R}$ we always have
\begin{eqnarray}\label{eq:3.11}
\tau(s)\leq\frac{2(b-a)}{\epsilon^2}.
\end{eqnarray}
Moreover, $|{A}_H(z(s+\tau(s)),\lambda(s+\tau(s)))|\leq \max\{|a|,|b|\}$ and
$\|\nabla {A}_H(z(s+\tau(s)),\lambda(s+\tau(s)))\|_{\mathcal{E}}\leq\epsilon$.
It follows from
Corollary~\ref{coro:3.1} that
\begin{eqnarray}\label{eq:3.12}
\|z(s+\tau(s))\|_{E_s}<C\quad \hbox{and}\quad |\lambda(s+\tau(s))|<C.
\end{eqnarray}
Since (\ref{met:3}) implies that
$((I+K)^{-1}w,w)_{\mathcal{E}}\geq\frac{2}{9}\|w\|^2_{\mathcal{E}}$ for any $w\in\mathcal{E}$,
we get
\begin{eqnarray*}
\|\nabla^K \mathcal{A}_H(z(t),\lambda(t))\|^2_{\mathcal{E}}&\le&\frac{9}{2}
\big(\nabla^K \mathcal{A}_H(z(t),\lambda(t)), (I+ K((z(t),\lambda(t)))^{-1} \nabla^K \mathcal{A}_H(z(t),\lambda(t))\big)_{\mathcal{E}}\\
&=&\frac{9}{2}\|\nabla^K \mathcal{A}_H(z(t),\lambda(t))\|^2_{g^K}.
\end{eqnarray*}
This and $(z^\prime(t),\lambda^\prime(t))=-\nabla^K \mathcal{A}_H(z(t),\lambda(t))$ lead to
\begin{eqnarray}\label{eq:3.13}
\int^{+\infty}_{-\infty}\|z^\prime(t)\|^2_{E_s}dt+\int^{+\infty}_{-\infty}
|\lambda^\prime(t)|^2dt&=&\int^{+\infty}_{-\infty}\|\nabla^K \mathcal{A}_H(z(t),\lambda(t))\|^2_{\mathcal{E}}dt\notag\\
&\leq& 5\int^{+\infty}_{-\infty}\|\nabla^K \mathcal{A}_H(z(t),\lambda(t))\|^2_{g^K}dt\notag\\
&\leq&5(b-a).
\end{eqnarray}
Using this, (\ref{eq:3.11})-(\ref{eq:3.12}) and  the H\"{o}lder inequality, we estimate
\begin{eqnarray}\label{eq:3.14}
\|z(s)\|&\leq& \|z(s+\tau(s))\|+\int^{s+\tau(s)}_s\|z^\prime(t)\|dt\notag\\
&\leq&C+\sqrt{5(b-a)}\sqrt{\tau(s)}\notag\\
&\leq&C+\frac{5(b-a)}{\epsilon}:=\frac{C_1}{\sqrt{2}}
\end{eqnarray}
and
\begin{eqnarray}\label{eq:3.15}
|\lambda(s)|&\leq& |\lambda(s+\tau(s))|+\int^{s+\tau(s)}_s|\lambda^\prime(t)|dt\notag\\
&\leq&C+\sqrt{5(b-a)}\sqrt{\tau(s)}\leq \frac{C_1}{\sqrt{2}}.
\end{eqnarray}
\qed


\subsection{Boundedness for the non-autonomous flow}\label{sec:3.3}

Suppose that ${K}\in C^0(\mathbb{R},\mathbf{K})$ and $H\in C^2(\mathbb{R}\times\Sigma M\oplus\Sigma M)$
satisfy
\begin{description}
\item[(i)] ${K}_t(\cdot):={K}(t,\cdot)$ is equal to $K_0$ for $t\le 0$, and $K_1$ for $t\ge 1$;
\item[(ii)] ${H}_t(\cdot):={H}(t,\cdot)$ is equal to $H_0$ for $t\le 0$, and $H_1$ for $t\ge 1$,
and satisfies $({\bf H}1)-({\bf H}2)$ and $({\bf H}4)$ for all $t\in[0,1]$;
\item[(iii)] there exists a constant $A$ such that
\begin{equation}\label{eq:3.3.1}
\int_{-\infty}^{+\infty}\int_M
\bigg|\frac{\partial H}{\partial t}(t,x,z(x))\bigg|dxdt\le A,\quad\forall z\in E_s.
\end{equation}
\end{description}

For such a pair $(H,K)$  we shall prove the boundedness on $\mathcal{E}$ of solutions of the following nonautonomous system
\begin{eqnarray} \label{eq:3.3.2}
\frac{dw(t)}{dt}+(I+K(t,w(t)))\nabla \mathcal{A}_H(t,x,w(t))=0.
\end{eqnarray}

\begin{proposition}\label{prop:3.3}
Fix a pair $(H,K)$ as above. Let  $w(t)=(z(t),\lambda(t))$ be any solution of (\ref{eq:3.3.2}) with $\lim\limits_{t\to\pm\infty}\mathcal{A}_{H_t}(z(t),\lambda(t))\in[a,b]$ for some $a<b$,
and let $\epsilon$ be as in Corollary~\ref{coro:3.1}. If $A<\epsilon/5$, then there exists a constant $C>0$ only depending on $a$ and $b$ such that $|\lambda(t)|<C$ and $\|z(t)\|_{E_s}<C$ for any $t\in\mathbb{R}$.
\end{proposition}
\proof
For a fixed $t$, we denote the gradient of $\mathcal{A}_{H_t}$ with respect to $g^{K_t}$ by $\nabla^{K_t}\mathcal{A}_{H_t}$. Since
\begin{equation}\label{eq:3.3.3}
\frac{d}{dt}\mathcal{A}_{H_t}(w(t))=-\|\nabla^{K_t}\mathcal{A}_{H_t}
(w(t))\|_{g^K}-\lambda(t)\int_M\frac{\partial H}{\partial t}(t,x,z(t))dx,
\end{equation}
we have
\begin{eqnarray} \label{eq:3.3.4}
&&\int^{+\infty}_{-\infty}\|\nabla^{K_t}\mathcal{A}_{H_t}
(w(t))\|_{g^K}^2dt\\
&\leq&\lim\limits_{t\to-\infty}\mathcal{A}_{H_t}(w(t))
-\lim\limits_{t\to+\infty}\mathcal{A}_{H_t}(w(t))+\int_M
\int^{+\infty}_{-\infty}|\lambda(t)|\bigg|\frac{\partial H}{\partial t}(t,x,z)\bigg|dt dx\notag\\
&\leq&b-a+A\|\lambda\|_\infty.
\end{eqnarray}
Defining $\tau(s)$ for $s\in\mathbb{R}$ as in the proof of Proposition~\ref{prop:3.2}, we have
\begin{eqnarray} \label{eq:3.3.5}
\epsilon^2\tau(s)&\leq&\int^{s+\tau(s)}_s\|\nabla\mathcal{A}_{H_t}
(w(t))\|^2_{\mathcal{E}}dt\notag\\
&\leq&2\int^{s+\tau(s)}_s\|\nabla^{K_t}\mathcal{A}_{H_t}
(w(t))\|_{g^K}^2dt\notag\\
&\leq&2(b-a+A\|\lambda\|_\infty),
\end{eqnarray}
and hence
\begin{equation}\label{eq:3.3.6}
\tau(s)\leq\frac{2(b-a+A\|\lambda\|_\infty)}{\epsilon^2}.
\end{equation}
Then as in (\ref{eq:3.13}) we have
\begin{eqnarray*}
\int^{+\infty}_{-\infty}\|z^\prime(t)\|^2_{E_s}dt+\int^{+\infty}_{-\infty}
|\lambda^\prime(t)|^2dt&=&\int^{+\infty}_{-\infty}\|\nabla^K \mathcal{A}_{H_t}(z(t),\lambda(t))\|^2_{\mathcal{E}}dt\notag\\
&\leq& 5\int^{+\infty}_{-\infty}\|\nabla^K \mathcal{A}_{H_t}(z(t),\lambda(t))\|^2_{g^K}dt\notag\\
&\leq&5(b-a+ A\|\lambda\|_\infty).
\end{eqnarray*}
This and (\ref{eq:3.3.6}) lead to
\begin{eqnarray*}
|\lambda(s)|&=&\Big|\lambda(s+\tau(s))+\int^{s+\tau(s)}_s
\lambda^\prime(t)dt\Big|\notag\\
&\leq&|\lambda(s+\tau(s))|+\int^{s+\tau(s)}_s\|\nabla\mathcal{A}^K_
{H_t}(w(t))\|_\mathcal{E}dt\notag\\
&\leq&C+\sqrt{2\tau(s)}\sqrt{5(b-a+A\|\lambda\|_\infty)}
\notag\\
&\leq&C+\frac{5(b-a+A\|\lambda\|_\infty)}{\epsilon}\quad\forall s\in\mathbb{R}.
\end{eqnarray*}
It follows from $A<\epsilon/5$ that
\begin{equation}\notag
\|\lambda\|_\infty\leq\frac{C\epsilon+5(b-a)}{\epsilon-5A}.
\end{equation}
Similarly, we have the uniform bound for $\|z(t)\|$. In fact, the following estimate holds
\begin{eqnarray*}
\|z(s)\|\leq\|z(s+\tau(s))\|+\frac{5(b-a+A\|\lambda\|_\infty)}{\epsilon}.
\end{eqnarray*}
From Corollary~\ref{coro:3.1}, we obtain $\|z(s+\tau(s))\|\leq C$ and thus $\|z(t)\|$ is uniformly bounded.
\qed

\subsection{Boundness in H\"{o}lder spaces}\label{sec:3.4}
In Proposition~\ref{prop:3.3} we get an uniform bound for all nonautonomous flows under certain conditions.
Following the ideas in \cite{AnV} we shall show that $z(t)$ is also uniformly bounded in H\"{o}lder spaces $C^{i,\alpha}(M,\Sigma M)$,
 $i=0,1$ (for the explicit definition see~\cite[Chapter 3]{Amm}) through the bootstrap argument.

\begin{proposition}\label{prop:3.4}
Assume $H$ is as in Proposition~\ref{prop:3.3}. Let $w(t)=(z(t),\lambda(t))$ be a solution of (\ref{eq:3.3.2}). Then $z(t)$ is bounded in $C^\alpha(M,\Sigma M)\oplus C^\beta(M,\Sigma M)$ with $0<\alpha<\min\{1,2s\}$ and $0<\beta<\min\{1,2(1-s)\}$. In particular, for some constant $C>0$ we have
 $$
\sup\limits_{-\infty<t<\infty}\|z(t)\|_{C^\alpha\oplus C^\beta}<C\sup\limits_{-\infty<t<\infty}\|z(t)\|_{E_s}.
$$
\end{proposition}

For  $0<\alpha<1$ (resp. $1<\alpha<2$),
let $C^\alpha(M,\Sigma M)$ represent H\"{o}lder spaces $C^{0,\alpha}(M,\Sigma M)$ (resp. $C^{1,\alpha-1}(M,\Sigma M)$).

In order to prove the above proposition we give two lemmas.

\begin{lemma}\label{lem:3.1}
Let $0<r<\infty$ and $1<\alpha,\beta<\infty$. Assume $u\in L^\alpha(M,\Sigma M)$.

 If $\alpha <\frac{n}{r}$ then $|D|^{-r}u\in L^\beta(M,\Sigma M)$ with $\frac{1}{\beta}=\frac{1}{\alpha}-\frac{r}{n}$.

 If $\alpha>\frac{n}{r}$ then $|D|^{-r}u\in C^{r-\frac{n}{\alpha}}(M,\Sigma M)$.
\end{lemma}

Since $\mathcal{D}_sL:E_s\to E_s$ is an self-adjoint Fredholm  isometry operator and  \begin{equation}\notag
(\mathcal{D}_sL)\circ(\mathcal{D}_sL)=\begin{pmatrix}
|D|^{-2s}D|D|^{-2(1-s)}D&0 \\
 0&|D|^{-2(1-s)}D|D|^{-2s}D
\end{pmatrix}={\rm Id}_{E_s},
\end{equation}
we can split $E_s$ into
\begin{equation}\label{eq:3.4.0}
E_s=E_+\oplus E_-
\end{equation}
 with $(\mathcal{D}_s L)|_{E_\pm}=\pm{\rm Id}_{E_\pm}$.
  Denote  by $P_\pm$  projections onto the eigenspace $E_\pm$ respectively. Then the evolution operator $\frac{d}{dt}+\mathcal{D}_sL$ has a fundamental solution operator
\begin{equation}\notag
G(t)=e^{-t}\chi_{\mathbb{R}_+}(t)P_--e^t\chi_{\mathbb{R}_-}(t)P_+,
\end{equation}
which gives us a formulation for any solution of (\ref{eq:3.3.2})
\begin{equation}\notag
z(t)=\int_{-\infty}^{+\infty}G(t-s)\Big\{\lambda(s)\mathcal{D}_s
H_z(s,x,z(s))+{\rm Pr_1}\big(K(s,z(s))\nabla\mathcal{A}_H)(w(s)\big)\Big\}ds,
\end{equation}
where ${\rm Pr_1}\big(K(s,z(s))\nabla\mathcal{A}_H(w(s))\big)$ is uniformly bounded in $C^2(M,\Sigma M\oplus\Sigma M)$ by (\ref{eq:2.12}) and Proposition~\ref{prop:3.3}.
To show that the operator $G(t)\mathcal{D}_s$ has a regularizing effect,
note that the above projections $P_\pm$ onto $E_\pm$
can be written as
\begin{equation}\notag
P_\pm=\frac{1}{2}\begin{pmatrix}
I_+&\pm|D|^{-2s}D \\
\pm|D|^{-2(1-s)}D& I_-
\end{pmatrix},
\end{equation}
where $I_\pm$ represents the identity operator on $E_\pm$ respectively. It is not hard to estimate
\begin{equation}\label{eq:3.4.1}
|G(t)\mathcal{D}_sH_z(t,x,z(t))|\leq\frac{1}{2}e^{-|t|}
\begin{pmatrix}
|D|^{-2s}|H_u|+|D|^{-1}|H_v|\\
|D|^{-1}|H_u|+|D|^{-2(1-s)}|H_v|
\end{pmatrix}.
\end{equation}
Using this estimate, the assumption $({\bf H}2)$ and Lemma~\ref{lem:3.1}, we obtain the bootstrap lemma

\begin{lemma}\label{lem:3.2}
Let $(z(t),\lambda(t))$ be a solution of (\ref{eq:3.3.2}) with $[z]_{\alpha,\beta}<\infty$, where
$$
z(t)=(u(t,\cdot), v(t,\cdot))\quad\hbox{and}\quad
[z]_{\alpha,\beta}:=\sup \limits_{-\infty<t<+\infty}\\\{\|u(t,\cdot)\|
 _{L^\alpha}+\|v(t,\cdot)\|_{L^\beta}\}.
 $$
  Suppose that two real numbers $\alpha^*$ and $\beta^*$ are defined by
 \begin{eqnarray} \label{eq:3.4.2}
\left\{ \begin{array}{l}
\frac{1}{\alpha^*}=\max\big\{\frac{p}{\alpha}-\frac{2s}{n},\frac{p(q+1)}
{\beta(p+1)}-\frac{2s}{n},\frac{q(p+1)}{\alpha(q+1)}-\frac{1}{n},
\frac{q}{\beta}-\frac{1}{n}\big\}\\[10pt]
\frac{1}{\beta^*}=\max\big\{\frac{p}{\alpha}-\frac{1}{n},
\frac{p(q+1)}{\beta(p+1)}-\frac{1}{n},
\frac{q(p+1)}{\alpha(q+1)}-\frac{2(1-s)}{n},
\frac{q}{\beta}-\frac{2(1-s)}{n}\big\}.
\end{array} \right.
\end{eqnarray}
If $\alpha^*,\beta^*>0$, then $[z]_{\alpha^*,\beta^*}<\infty$.\\
If $\alpha^*<0$, then $u(t,\cdot)$ is uniformly bounded in $C^{-\frac{1}{\alpha^*}}(M,\Sigma M)$.\\
If $\beta^*<0$, then $v(t,\cdot)$ is uniformly bounded in $C^{-\frac{1}{\beta^*}}(M,\Sigma M)$.
\end{lemma}

\noindent\textbf{Proof of Proposition~\ref{prop:3.4}}. Since $z(t)$ is uniformly bounded in $E_s$, we have $[z]_{\alpha_1,\beta_1}<\infty$ with $\alpha_1=\frac{2n}{n-2s}$ and $\beta_1=\frac{2n}{n-2(1-s)}$ because of the Sobolev embedding theorems.
By Lemma~\ref{lem:3.2} the condition $[z]_{\alpha,\beta}<\infty$ implies $[z]_{\alpha^*,\beta^*}<\infty$. It follows from the subcritical condition (\ref{eq:1.4}) that $\alpha^*>\alpha$ and $\beta^*>\beta$ whenever $\alpha\geq\alpha_1$ and $\beta\geq\beta_1$. After $k$ times of iterations one gets $[z]_{\alpha_k,\beta_k}<\infty$ with $\alpha_{k+1}<0$ and $\beta_{k+1}<0$. Using the bootstrap lemma again  we arrive at the conclusion of Proposition~\ref{prop:3.4}.
\qed

\section{ Relative index and Moduli space of trajectories }\label{sec:4}
\setcounter{equation}{0}
\subsection{Morse functions and the relative index}\label{sec:4.1}

According to the usual definition of the Morse index of a critical point,
 the Rabinowitz-Floer functional $\mathcal{A}_H$ has infinite index and co-index at each critical point.
  In order to give a natural grading of the Rabinowitz-Floer homology groups
   we will adopt the way of Abbondandolo \cite{AbM}.

If $(z,\lambda)$ is a critical of $\mathcal{A}_H$, then the Hessian of $\mathcal{A}_H$ at $(z,\lambda)$ can be written as
\begin{equation}\notag
{\rm Hess}\mathcal{A}_H(z,\lambda)=\begin{pmatrix}
\mathcal{D}_s L-\lambda\mathcal{D}_sH_{zz}(x,z)&-\mathcal{D}_sH_z(x,z) \\
-\big(\mathcal{D}_sH_z(x,z)\big)^*&0
\end{pmatrix}.
\end{equation}
\begin{definition}\label{def:4.1}
{\rm A critical point $(z,\lambda)$ of $\mathcal{A}_H$ is called  {\it nondegenerate}  if the Hessian operator ${\rm Hess}\mathcal{A}_H(z,\lambda):\mathcal{E}\to \mathcal{E}$
is a bounded linear bijective operator. The functional $\mathcal{A}_H$
is said to be \emph{Morse} if all critical points of it are nondegenerate.}
\end{definition}

\begin{remark}
{\rm In our case, we always view $\mathcal{E}$ as a real Hilbert space with the inner product
$(\cdot,\cdot)_\mathcal{E}$ defined by (\ref{eq:2.5}). In the following, for a subspace $V$ of $\mathcal{E}$, ${\rm dim}(V)$ stands for the real dimension of $V$.}
\end{remark}

\begin{definition}\label{def:4.2}
{\rm Let $U$ and $V$ be two closed subspaces in a Hilbert space $H$. Denote $P_U$ and $P_V$ the orthogonal projections onto $U$ and $V$ respectively. If $P_U-P_V$ is compact on $H$,  $U$ is said
to be {\it commensurable} to $V$, and define the relative dimension of $U$ and $V$ as
$${\rm dim}(U,V)={\rm dim}(U\cap V^\perp)-{\rm dim}(U^\perp\cap V).$$}
\end{definition}

One can check that it is well-defined and finite.
Moreover, if $U$, $V$ and $W$ are three each other commensurable closed subspaces in $H$,
then
\begin{equation}\label{eq:4.1.1}
{\rm dim}(U,V)={\rm dim}(U, W)+{\rm dim}(W,V)
\end{equation}
(cf. \cite[Section 2]{AbM}).

\begin{lemma}\label{lem:4.1}
Suppose $A$ and $B$ are self-adjoint Fredholm operators on a Hilbert space $H$ with $A-B$ compact. Then the negative (resp. positive ) eigenspaces of $B$ is commensurable with the negative (resp. positive ) eigenspaces of $A$.
\end{lemma}

\begin{definition}\label{def:4.3}
{\rm Let $E_-$ be as in (\ref{eq:3.4.0}). Denote $V^-(z,\lambda)$ the maximal negative definite subspace in $\mathcal{E}$ of the Hessian of $\mathcal{A}_H$ at a critical point $(z,\lambda)$. We define the relative index as
\begin{equation}\label{eq:4.1.2}
i_{\rm rel}(z,\lambda)={\rm dim}(V^-(z,\lambda),E_-\times\{0\}).
\end{equation}}
\end{definition}
\begin{lemma}\label{lem:4.2}
Under the assumptions $({\bf H}1)-({\bf H}4)$, the relative index is well defined for all critical points of $\mathcal{A}_H$.
\end{lemma}
\proof
Let $(z,\lambda)$ be a critical point of $\mathcal{A}_H$. Notice that $\mathcal{D}_sL\oplus Id_{\mathbb{R}}$ and ${\rm Hess}\mathcal{A}_H(z,\lambda)$ are both self-adjoint Fredholm operators on $\mathcal{E}$ respectively, where $Id_{\mathbb{R}}$ is the identity operator on $\mathbb{R}$. Moreover, $E_-\times\{0\}$ and $V^-(z,\lambda)$ are the negative eigenspaces of $\mathcal{D}_sL\oplus Id_{\mathbb{R}}$ and ${\rm Hess}\mathcal{A}_H(z,\lambda)$, respectively. It follows from Proposition~\ref{prop:2.1} that
\begin{equation}\label{eq:4.3}
{\rm Hess}\mathcal{A}_H(z,\lambda)-\mathcal{D}_sL\oplus Id_{\mathbb{R}}=\begin{pmatrix}
-\lambda\mathcal{D}_sH_{zz}(x,z)&-\mathcal{D}_sH_z(x,z) \\
-\big(\mathcal{D}_sH_z(x,z)\big)^*&-1
\end{pmatrix}
\end{equation}
is compact on $\mathcal{E}$. Hence Lemma~\ref{lem:4.1} implies that the relative index
$i_{\rm rel}(z,\lambda)$ is well defined.
\qed

\subsection{Moduli space of trajectories and grading}\label{sec:4.2}

Let $w_1=(z_1,\lambda_1)$ and $w_2=(z_2,\lambda_2)$ be two nondegenerate critical points of $\mathcal{A}_H$,
and let $a$ and $b$ be two real numbers  such that $a\leq\mathcal{A}_H(w_i)\leq b$, $i=1,2$. Define  the space of connecting orbits from $w_1$ to $w_2$ as
\begin{eqnarray} \label{eq:4.2.1}
\mathcal{M}_{H,K}^{a,b}(w_1,w_2)=\left\{w\in C^1(\mathbb{R},\mathcal{E})\Bigg|\begin{array}{l}
w(t)\;\hbox{ is a solution of (\ref{eq:2.11}) with}\\  w_1=w(-\infty)=:\lim_{t\to-\infty}w(t)\\
 \hbox{ and } w_2=w(+\infty)=:\lim_{t\to+\infty}w(t)
\end{array} \right\},
\end{eqnarray}
and  the moduli space of trajectories $\widehat{\mathcal{M}}_{H,K}^{a,b}(w_1,w_2)$ to be the quotient of $\mathcal{M}_{H,K}^{a,b}(w_1,w_2)$ with respect to the free action of $\mathbb{R}$ given by the
 flow of $-\nabla^K\mathcal{A}_H$. Denote by $\hat{w}\in \widehat{\mathcal{M}}_{H,K}^{a,b}(w_1,w_2)$ the unparametrized trajectory corresponding to ${w\in\mathcal{M}}_{H,K}^{a,b}(w_1,w_2)$. Let
 $$
 Q_1=\overline{w}+W^{1,2}(\mathbb{R},\mathcal{E})\quad\hbox{and}\quad
 Q_0=L^2(\mathbb{R},\mathcal{E}),
 $$
  where $\overline{w}:\mathbb{R}\to \mathcal{E}$ is a fixed smooth map which satisfies
  $\overline{w}(t)\equiv w_1$ for $t\leq -1$ and $\overline{w}(t)\equiv w_2$ for $t\geq 1$.
The space of parametrized trajectories $\mathcal{M}_{H,K}^{a,b}(w_1,w_2)$ can be considered as the zero
set of the map $\mathcal{F}_{H,K}:Q_1\to Q_0$ given by
\begin{eqnarray} \label{eq:4.2.2}
\mathcal{F}_{H,K}(w)=\frac{d w}{dt}+\nabla^K \mathcal{A}_H(w).
\end{eqnarray}
Suppose that $\mathcal{A}_H$ is a Morse function. To determine the dimension of the moduli space of trajectories, we need that $(H,K)$ satisfies the so-called  \emph{Morse-Smale condition}, that is to say, for every pair of critical points $w_1$, $w_2$ of $\mathcal{A}_H$ the unstable manifold $W^u(w_1;-\nabla^K \mathcal{A}_H)$ is transverse to the stable manifold $W^s(w_2;-\nabla^K\mathcal{A}_H)$, i.e.,
 $$T_{w}W^u(w_1)+T_{w}W^s(w_2)=\mathcal{E},\quad \forall w\in W^u(w_1)\cap W^s(w_2).$$
The pair $(H,K)$ with $H$ satisfying $({\rm H1})-({\rm H4})$ and $K\in\mathbf{K}$ is called \emph{regular} if $\mathcal{A}_{H}$ is Morse and $(H,K)$ satisfies the Morse-Smale condition. Denote by $\Omega_{reg}$ the set consisting of all such regular pairs.

\begin{lemma}[{\rm \cite[Proposition 12]{AnV}}]\label{lem:4.2.1}
Let $E$ be a Banach space. Let $B:\mathbb{R}\to \mathscr{L}(E,E)$ be a continuous map such that
each $B(t)$ is  compact operator and that $\lim_{t\to\pm\infty}B(t)=0$. Then the operator $M_B:W^{1,2}(\mathbb{R},E)\to L^2(\mathbb{R},E)$ given by
$$
M_Ba(t)=B(t)a(t)\quad\hbox{for}\; a\in W^{1,2}(\mathbb{R},E)
$$
is compact.
\end{lemma}

\begin{lemma}[{\rm \cite[Theorem 3.4]{AbM}}]\label{lem:4.2.2}
Let $E$ be a Banach space. Let $A:\mathbb{R}\to L(E,E)$ be continuous map such that
$$
A(-\infty):=\lim\limits_{t\to-\infty}A(t)\quad\hbox{and}\quad A(+\infty):=\lim\limits_{t\to+\infty}A(t)
$$
exist and are self-adjoint and invertible. Then the operator
$$
L_A:W^{1,2}(\mathbb{R},E)\to L^2(\mathbb{R},E)
$$
given by $L_Aw(t)=w'(t)+A(t)w(t)$,
is a Fredholm operator of index
$$
{\rm ind}L_A={\rm dim}\bigl(V^-(A(-\infty)),V^-(A(+\infty))\bigr),
$$
where $V^-(L)$ represents the negative eigenspace of $L$ on $E$.
\end{lemma}

\begin{proposition}\label{prop:4.2.1}
 Let $(H,K)$ be a regular pair. Suppose that the relative indexes of two
 critical points $w_1$ and $w_2$ of $\mathcal{A}_H$ satisfy  $i_{\rm rel}(w_1)-i_{\rm rel}(w_2)>0$.
  Then $\widehat{\mathcal{M}}_{H,K}^{a,b}(w_1, w_2)$ is a manifold of dimension $i_{\rm rel}(w_1)-i_{\rm rel}(w_2)-1$ if ${\mathcal{M}}_{H,K}^{a,b}(w_1, w_2)\ne\emptyset$.
\end{proposition}
\proof
Let $K_\theta(w)=I+\theta K(w)$ for any $w\in \mathcal{E}$, where $\theta\in[0,1]$.
Then the linearized operator of $\mathcal{F}_{H,\theta K}$ at $w=(z,\lambda)$ is given by
\begin{equation}\label{eq:4.2.3}
D\mathcal{F}_{H,\theta K}(w)=\frac{d}{dt}+ K_\theta(w(t)){\rm Hess}\mathcal{A}_H(w(t))
+\{dK_\theta(w(t))[\cdot]\}\nabla\mathcal{A}_H(w(t)).
\end{equation}
We shall use the implicit function theorem to prove our result. To this end we need to show that $D\mathcal{F}_{H,K}(w)$ is Fredholm and onto. Notice that ${\rm Hess}\mathcal{A}_H(w)$ has the following decomposition
\begin{equation}\notag
\begin{array}{l}
{\rm Hess}\mathcal{A}_H(z,\lambda)=
\begin{pmatrix}
\mathcal{D}_sL&0 \\
0&1
\end{pmatrix}
+
\begin{pmatrix}
-\lambda\mathcal{D}_sH_{zz}(x,z)&-\mathcal{D}_sH_z(x,z) \\
-\big(\mathcal{D}_sH_z(x,z)\big)^*&-1
\end{pmatrix}.
\end{array}
\end{equation}
It follows from Proposition~\ref{prop:2.1} that ${\rm Hess}\mathcal{A}_H(w)$ is a compact perturbation of $\scriptscriptstyle\begin{pmatrix}\begin{smallmatrix}\mathcal{D}_sL&0 \\0&1\end{smallmatrix}\end{pmatrix}$ with spectrum $\{-1,1\}$. For a fixed $t\in\mathbb{R}$, by definition the linear maps
$$
T(t):\mathcal{E}\to C^2(M,\Sigma M\oplus\Sigma M)\times \mathbb{R},\quad
\xi\mapsto\{dK_\theta(w(t))[\xi]\}\nabla\mathcal{A}_H(w(t))
$$
is bounded, and hence compact as an operator from $\mathcal{E}$ to $\mathcal{E}$. Since $\lim_{t\to\pm\infty}\nabla\mathcal{A}_H(w(t))=0$, Lemma~\ref{lem:4.2.1} implies that the third term in (\ref{eq:4.2.3}) is compact. Therefore $D\mathcal{F}_{H,\theta K}(w)$ is a family of Fredholm  operator from $W^{1,2}(\mathbb{R},\mathcal{E})$ to $L^2(\mathbb{R},\mathcal{E})$. By the homotopy invariance of Fredholm index (cf. \cite{Kat}) and Lemma~\ref{lem:4.2.2}, we obtain
\begin{eqnarray*}
{\rm ind}(D\mathcal{F}_{H,K}(w))&=&{\rm ind}(D\mathcal{F}_{H,0}(w))\notag\\
&=&{\rm dim}(V^-(w_1),E_-\times\{0\})+{\rm dim}(E_-\times\{0\},V^-(w_2))\notag\\
&=&{\rm dim}(V^-(w_1),E_-\times\{0\})-{\rm dim}(V^-(w_2), E_-\times\{0\})\notag\\
&=&i_{\rm rel}(w_1)-i_{\rm rel}(w_2).\notag
\end{eqnarray*}

On the other hand, the Morse-Smale condition guarantees that $D\mathcal{F}_{H,K}(w)$ is onto (see \cite{Abb} for a detailed proof). So $\mathcal{M}_{H,K}^{a,b}(w_1,w_2)$ is a manifold of dimension $i_{\rm rel}(w_1)-i_{\rm rel}(w_2)$ and thus  the desired result is obtained by modeling the free $\mathbb{R}$-action. \qed

\subsection{Broken trajectories and gluing}\label{sec:4.3}
By Proposition~\ref{prop:3.4}, for any $w=(z,\lambda)\in\mathcal{M}_{H,K}(w_1,w_2)$ with $\sup_t|\mathcal{A}_H(w(t))|<\infty$, $\lambda(t)$ is bounded in $\mathbb{R}$ and $z(t)$ belongs to a compact subset of
$$
X_{\alpha,\beta}:=C^\alpha(M,\Sigma M)\oplus C^\beta(M,\Sigma M)
$$
for some two constants $0<\alpha<\min\{1,2s\}, 0<\beta<\min\{1,2(1-s)\}$.
 As a result, the moduli spaces are modeled on the affine space
 $$
 \mathcal{Q}^1:=C^1(\mathbb{R},X),\quad\hbox{where}\; X=X_{\alpha,\beta}\times\mathbb{R}.
 $$
 In the following, we will equip the $C^0_{loc}(\mathbb{R},X)$- topology on the moduli space $\widehat{\mathcal{M}}_{H,K}(w_1,w_2)$, i.e., the uniform $X$-norm convergence on bounded intervals of $\mathbb{R}$.

\begin{definition}
{\rm A broken trajectory from $x\in{\rm Crit}(\mathcal{A}_H)$ to $y\in{\rm Crit}(\mathcal{A}_H)$ consists of critical points $x=w_0,w_1,\ldots,w_k=y$, $k\geq2$ and a tuple $(\hat{u}_1,\ldots,\hat{u}_k)$ of unparametrized flow lines of $-\nabla^K\mathcal{A}_H$ such that $u_{i-1}(-\infty)=w_{i-1}$ and $u_{i-1}(+\infty)=w_i$, $1\leq i\leq k$.}
\end{definition}

The evaluation map $EV:\mathcal{M}_{H,K}(w_1,w_2)\to X$ defined by
$$EV(w)=w(0)$$
is a continuous injective map. It has also a precompact image set. In fact, let $(w^i)_{i=1}^\infty$ be a sequence in $\mathcal{M}_{H,K}(w_1,w_2)$. Write  $EV (w^i)=w^i_0$ and denote the corresponding orbits by
$c^i:=w^i(\mathbb{R})$.
As a result of the compactness before, the union set $\bigcup_i c^i\bigcup\{w_1,w_2\}$ constitutes a compact set in $X$.  Under the assumption that the critical points of $\mathcal{A}_H$ are isolated,
a standard argument (\cite[Section 5]{AnV}) shows that
 the sequence $\bar{c}^i=c^i\bigcup\{w_1,w_2\}$ has a subsequence converging to some compact set $c^*$ in $X$ which is either a broken trajectory or the union of an unparametrized flow line and the set $\{w_1,w_2\}$.

The above argument tells us that moduli spaces of trajectories are generally not compact. In the following we will use a gluing construction similar to that in Morse homology \cite[Chapter 2]{ScM} to show that the closure of $EV\bigl(\mathcal{M}_{H,K}(w_1,w_2)\bigr)$ is compact in $X$.

Suppose that $w_{12}\in\mathcal{M}_{H,K}(w_1,w_2)$ and $w_{23}\in\mathcal{M}_{H,K}(w_2,w_3)$ with $i_{rel}(w_1)>i_{rel}(w_2)>i_{rel}(w_3)$.
Since both $D\mathcal{F}_{H,K}(w_{12})$ and $D\mathcal{F}_{H,K}(w_{23})$ are Fredholm and onto provided that $(H,K)$ satisfies the Morse-Smale condition, they have bounded right inverses $\Psi_{12}$ and $\Psi_{23}$ from $L^2(\mathbb{R},\mathcal{E})$ to $W^{1,2}(\mathbb{R},\mathcal{E})$, respectively. Let $\zeta$ be a nonnegative smooth function such that
$$
\zeta(t)\equiv0\quad\hbox{for}\; t\leq-1,\quad\zeta(t)\equiv1\quad\hbox{for}\; t\geq 1,\quad\hbox{and}\; \zeta^\prime(t)\geq0\quad\forall t.
$$
 Let $\zeta_T(t)=\zeta(\frac{t}{T})$ for large $T>0$.
 Now we glue $w_{12}$ and $w_{23}$ as follow:
  $$
  w_{13,T}(t)=(1-\zeta_T(t))w_{12}(t+2T)+\zeta_T(t)w_{23}(t-2T),
 \hspace{2mm}t\in \mathbb{R}.
 $$
Accordingly, we define the gluing operator
$$\Phi_T=\rho_T^+\tau_{2T}\Psi_{12}\tau_{-2T}\rho_T^+
+\rho_T^-\tau_{2T}\Psi_{23}\tau_{-2T}\rho_T^-,$$
where $\tau_s$ is a translation operator satisfying $\tau_hg(t)=g(t+h)$ and $\rho^\pm$ is a pair of smooth functions satisfying
\begin{eqnarray*}
&&\rho_T^\pm=\rho^\pm\bigg(\frac{t}{T}\bigg),\quad
\hspace{2mm}(\rho_1^+)^2+(\rho_1^-)^2=1,\\
&&\rho_1^+(t)=0\;\text{ for }    t\leq-1,\hspace{2mm}\rho_1^-(t)=\rho_1^+(-t).
\end{eqnarray*}
A direct computation yields that $D\mathcal{F}_{H,K}(w_{13,T})\circ\Phi_T$ converges to the identity operator on $L^2(\mathbb{R},\mathcal{E})$ as $T\to\infty$.
Then it follows from the implicit function theorem that the equation $$\mathcal{F}_{H,K}(w)=0$$
has solutions of the form $w=w_{13,T}+\Phi_T\eta$ with $\eta\in L^2(\mathbb{R},\mathcal{E})$ whenever $T$ is large enough. Moreover, such a solution $w$ with small $\eta\in L^2(\mathbb{R},\mathcal{E})$
is unique, denoted by $w_{12}\#_{\scriptscriptstyle{T}}w_{23}$.
Hence one can approximate the broken trajectory $(\hat{w}_{12},\hat{w}_{23})$ by such $C^1$-glued orbits. Since these glued orbits represent approximate solutions of
 the negative gradient flow in a unique way for large gluing parameter $T$,
    from the fact that $EV\bigl(\mathcal{M}_{H,K}(w_1,w_2)\bigr)$ has compact
   closure we deduce that the moduli space $\widehat{\mathcal{M}}_{H,K}(w_1,w_2)$ is
    finite whenever $i_{\rm rel}(w_1)=i_{\rm rel}(w_2)+1$.

\section{The Rabinowitz Floer complex of $\mathcal{A}_H$}\label{sec:5}
\setcounter{equation}{0}

In former two subsections we shall give the construction of the
Rabinowitz Floer homology of $\mathcal{A}_H$ in Morse situation and Morse-Bott one.
Then we prove continuation invariance of the homology in the Morse case in Section~5.3.

\subsection{The Morse situation}\label{sec:5.0}

Given a pair $(H,K)\in\Omega_{reg}$, $k\in\mathbb{N}$ and real numbers $a<b$, let ${\rm Crit}(\mathcal{A}_H)$ be the set of critical points of the functional $\mathcal{A}_H$, ${\rm Crit}_k(\mathcal{A}_H)=\{x\in {\rm Crit}(\mathcal{A}_H)\,|\,i_{\rm rel}(x)=k\}$ and
\begin{eqnarray*}
&&{\rm Crit}^{[a,b]}(\mathcal{A}_H):=\{x\in {\rm Crit}(\mathcal{A}_H)\;|\;a\le\mathcal{A}_H(x)\le b\},\\
&&{\rm Crit}_k^{[a,b]}(\mathcal{A}_H):=\{x\in {\rm Crit}_k(\mathcal{A}_H)\;|\;a\le\mathcal{A}_H(x)\le b\}.
\end{eqnarray*}
Denote by ${\rm CF}_k(H)$ the chain complex as the vector space over $\mathbb{Z}_2$ generated by
${\rm Crit}_k(\mathcal{A}_H)$. It  needs not to be finitely
generated.  But its $\mathbb{Z}_2$-subspace ${\rm CF}_k^{[a,b]}(H)$
generated by ${\rm Crit}_k^{[a,b]}(\mathcal{A}_H)$ has only finite elements
 since  $\mathcal{A}_H$ is Morse and satisfies the $(PS)_c$ condition. As we said before, if $x,y\in {\rm Crit}(\mathcal{A}_H)$
satisfy  $i_{\rm rel}(x)=i_{\rm rel}(y)+1$, then the integer $\#\widehat{\mathcal{M}}_{H,K}(x,y)$
is finite. Define
$$
n_2(x,y):=\#\widehat{\mathcal{M}}_{H,K}(x,y)\text{ mod }[2],
$$
and the boundary operator $\partial_k:{\rm CF}_k^{[a,b]}(H)\to{\rm CF}_{k-1}^{[a,b]}(H)$ by
\begin{equation}\label{e:BH}
\partial_k \left(\sum^r_{j=1}m_jx_j\right)=\sum^r_{j=1}\sum\limits_{y\in{\rm CF}_{k-1}^{[a,b]}(H)}
m_jn_2(x_j,y)y
\end{equation}
for $\sum^r_{j=1}m_jx_j\in {\rm CF}_k^{[a,b]}(H)$.
Due to the gluing argument of Section~\ref{sec:4.3}, we have $\partial_k\partial_{k-1}=0$. So $({\rm CF_*}^{[a,b]}(H),\partial_*)$ is a chain complex which is called the \emph{Rabinowitz Floer complex} of $\mathcal{A}_H$. The corresponding homology is called the \emph{Rabinowitz Floer homology} of $\mathcal{A}_H$ defined by
$$RHF_k^{[a,b]}(H,K)=\frac{{\rm ker}(\partial_k)}{{\rm im}(\partial_{k+1})}.$$

\subsection{The Morse-Bott situation}\label{sec:5.1}
In this subsection, we construct the Rabinowitz Floer homology when the functional $\mathcal{A}_H$ is Morse-Bott.
\begin{definition}
{\rm Let $B$ be a Hilbert space. A $C^k$ functional $f$  ($k\geq 2$) on a Hilbert space $B$ is called {\it Morse-Bott} if its critical set $${\rm Crit}(f)=\{x\in B|df(x)=0\}$$ is a $C^{k-1}$ submanifold (possible with components of different dimensions) and  it holds that
$$
T_x{\rm Crit}(f)={\rm ker}({\rm Hess}(f)(x))\quad\forall x\in{\rm Crit}(f).
$$}
\end{definition}

 Suppose that $H$ is invariant under the action of $S^1$ on $\Sigma
M\times \Sigma M$, that is,
$$
H(x,e^{i\theta}z)=H(x,z), \hspace{2mm} \forall \theta\in\mathbb{R},
\hspace{2mm} \forall x\in M, \hspace{2mm} \forall z=(u,v)\in\Sigma_x
M\times \Sigma_x M.
$$
By extending the $S^1$ action on $\mathbb{R}$ trivially, we find that $\mathcal{A}_H$ is also invariant under this action, i.e.,
$$
\mathcal{A}_H(e^{i\theta}z,e^{i\theta}\lambda)=
\mathcal{A}_H(z,\lambda),\hspace{2mm} \forall z\in E_s, \hspace{2mm} \forall \lambda\in\mathbb{R}.
$$
In this case, the functional $\mathcal{A}_H$ is never Morse. One way to overcome it is to choose an additional small perturbation to reduce to the Morse situation as before. However,
even for the nonlinearity $H(x,u,v)=\frac{1}{2}(u^2+v^2)$,  giving
 an explicit computation of  the Rabinowitz Floer homology of it  needs an elaborate perturbation and a good computation for the index. It seems to be difficult for the authors to carry out them.

Instead we may assume that $\mathcal{A}_H$ is Morse-Bott, and then choose a Morse function satisfying the Morse-Smale condition on the critical manifold. The chain complex is generated by the critical points of this Morse function, while the boundary operator is defined by counting flow lines with cascades. For the definition of gradient flow lines with cascades we refer to Frauenfelder \cite{Fra} or \cite{BoO,CiF}.

\noindent{\bf 5.2.1. Flow lines with cascades.}\quad Assume that
 ${\rm Crit}(\mathcal{A}_H)$ is a $C^2$-submanifold of $\mathcal{E}$. We choose a Riemannian metric $g$ on ${\rm Crit}(\mathcal{A}_H)$ and a Morse function $h:{\rm Crit}(\mathcal{A}_H)\to\mathbb{R}$, such that $(h,g)$ satisfy the Morse-Smale condition, that is,
 for every pair of critical points $x$, $y$ of $h$, the unstable manifold $W^u(x;-\nabla h)$ is transverse to the stable manifold $W^s(x;-\nabla h)$.
 Let ${\rm ind}(x)$ denote the Morse index of $h$ at $x$, i.e., ${\rm ind}(x)={\rm dim} W^u(x;-\nabla h)$.
 We now assign an index $\nu$ to $x$ by
$$
\nu(x):=i_{\rm rel}(x)+{\rm ind}(x).
$$

Given two real numbers $a<b$,   we further assume
  \begin{center}
  \textsf{the critical manifold of $\mathcal{A}_H$ containing in $\mathcal{A}_H^{-1}[a, b]$
    is the union of a\\ finite number of disjoint, compact, non-degenerate critical manifolds.}
  \end{center}
  Let $\Psi_h(t)\in{\rm Diff}({\rm Crit}^{[a,b]}(\mathcal{A}_H))$ be the smooth family of diffeomorphisms defined by
$$\Psi_h(t)(p)=\varphi_p(t),\quad\forall~p\in{\rm Crit}^{[a,b]}(\mathcal{A}_H),$$
where $\varphi_p$ is the flow line of $-\nabla h$ with $\varphi_p(0)=p$.
 Let  $\sigma_0$ and $\sigma_1$ be two components of ${\rm Crit}^{[a,b]}(\mathcal{A}_H)$,
 and $x_i\in\sigma_i$, $i=0,1$, and an integer $m\geq1$. A \emph{flow line from $x_0$ to $x_1$ with $m$ cascades} consists of $m-1$ components $c_1,\ldots,c_{m-1}$ of ${\rm Crit}(\mathcal{A}_H)$ and a $m$ tuple $(w_1,\ldots,w_m)$ of solutions of the Rabinowitz-Floer equation (\ref{eq:2.11}) such that
\begin{eqnarray*}
&&w_1(-\infty)\in W^u(x_0;-\nabla h),\hspace{4mm}
w_m(+\infty)\in W^s(x_1;-\nabla h),\\
&&w_{j+1}(-\infty)=\Phi_h(t_j)(w_j(+\infty))\quad\hbox{for some}\;t_j\in \mathbb{R}_+,\;j=
1,\cdots,m-1,
\end{eqnarray*}
where $\mathbb{R}_+:=\{r\in\mathbb{R}:r\geq0\}$. Denote by
$$
\mathcal{M}^{a,b}_m(x_0,x_1;H,K,h,g)
$$
the set of flow fines from $x_0$ to $x_1$ with $m$ cascades, and by
 $\widehat{\mathcal{M}}^{a,b}_m(x_0,x_1;H,K,h,g)$ the quotient of $\mathcal{M}^{a,b}_m(x_0,x_1;H,K,h,g)$ by the free action of $\mathbb{R}^m$ given by $s$-translations on each $w_j$.

We also define the set of \emph{flow lines from $x_0$ to $x_1$ with zero cascades}, denoted by
$$
\mathcal{M}_0^{a,b}(x_0,x_1;H,K,h,g),
$$
is the intersection  $W^u(x_0;-\nabla h)\cap W^s(x_1;-\nabla h)$ whenever $\sigma_0=\sigma_1$,
and empty if $\sigma_0$ and $\sigma_1$ are different.
 When $\sigma_0=\sigma_1,x_0\neq x_1$ the quotient of $\mathcal{M}_0^{a,b}(x_0,x_1;H,K,h,g)$ by the $\mathbb{R}$- action is denoted by $\widehat{\mathcal{M}}_0^{a,b}(x_0,x_1;H,K,h,g)$, otherwise $\widehat{\mathcal{M}}_0^{a,b}(x_0,x_1;H,K,h,g)$ is defined to be empty.

We call
$$\widehat{\mathcal{M}}^{a,b}(x_0,x_1;H,K,h,g)=\bigcup\limits_{m\geq0}
\widehat{\mathcal{M}}_m^{a,b}(x_0,x_1;H,K,h,g)$$
the moduli space of \emph{Morse-Bott trajectories with cascades}.
For each $m\geq0$, the moduli space of trajectories $\widehat{\mathcal{M}}_m^{a,b}(x_0,x_1;H,K,h,g)$ with $m$ cascades can be compactified by adding broken trajectories in the same way as Theorem A.10 in \cite{Fra}. The rest of this subsection is mainly devoted to prove

\begin{theorem}\label{th:5.1}
 Given two components of ${\rm Crit}^{[a,b]}(\mathcal{A}_H)$,
 $\sigma_0$, $\sigma_1$,  and $x_i\in\sigma_i$, $i=0,1$, for generic choice of $K$, the set $\widehat{\mathcal{M}}^{a,b}(x_0,x_1;H,K,h,g)$ has the structure of a finite dimensional manifold and its dimension is given by
$${\rm dim}\widehat{\mathcal{M}}^{a,b}(x_0,x_1;H,K,h,g)
=\nu(x_0)-\nu(x_1)-1.$$
\end{theorem}

 The idea of the proof of the theorem is owned to Frauenfelder \cite[Theorem A.11]{Fra}.
By (\ref{eq:4.3}), we can write ${\rm Hess}\mathcal{A}_H(z,\lambda)=T_1+T_2,$ where
$$
T_1=\mathcal{D}_sL\oplus Id_{\mathbb{R}},\quad (T_1)^2=Id_\mathcal{E},
$$
and
\begin{eqnarray*}T_2=\begin{pmatrix}
-\lambda\mathcal{D}_sH_{zz}(x,z)&-\mathcal{D}_sH_z(x,z) \\
-\big(\mathcal{D}_sH_z(x,z)\big)^*&-1
\end{pmatrix}
\end{eqnarray*}
is compact on $\mathcal{E}$. Since  $(T_1+T_2)^2=Id_\mathcal{E}+T_1T_2+T_2T_1+T_2^2$, and $T_1T_2+T_2T_1+T_2^2$ is a compact self-adjoint operator, by the spectral theory of compact operator
we obtain that the spectrum of $(T_1+T_2)^2$ can be described as a sequence of discrete points tending to $1$, and thus ${\rm Spec}({\rm Hess}\mathcal{A}_H(z,\lambda))\setminus\{0\}$ is away from $0$.
Let $d$ be a constant such that
\begin{equation}\label{eq:5.1.1}
0<d<\min\{|\lambda|:\lambda\in{\rm Spec}({\rm Hess}(\mathcal{A}_H)(w))\setminus\{0\},~w\in{\rm Crit}^{[a,b]}(\mathcal{A}_H)\},
\end{equation}
and let $s_d:\mathbb{R}\to\mathbb{R}$ be the weight function defined by $$s_d(t)=e^{d\vartheta(t)t},$$
where $\vartheta(t)$ is a smooth cutoff function satisfying $\vartheta(t)=-1$ for $t<0$ and $\vartheta(t)=1$ for $t>1$.
We  introduce suitable weighted Sobolev spaces as follow:
\begin{eqnarray*}
&&L^2_d(\mathbb{R},\mathcal{E}):=\left\{w\in L^2(\mathbb{R},\mathcal{E})\,|\, s_dw\in L^2(\mathbb{R},\mathcal{E})\right\},\\
&&W^{1,2}_d(\mathbb{R},\mathcal{E}):=\left\{w\in W^{1,2}(\mathbb{R},\mathcal{E})\,|\, s_dw\in W^{1,2}(\mathbb{R},\mathcal{E})\right\}
\end{eqnarray*}
with weighted norms
\begin{eqnarray*}
&&\|w\|_d:=\|s_dw\|_{L^2(\mathbb{R},\mathcal{E})},\quad w\in L^2_d(\mathbb{R},\mathcal{E}),\\
&&\|w\|^2_{1,d}:=\|s_dw\|^2_{L^2(\mathbb{R},\mathcal{E})}+\|s_dw'\|^2_{L^2(\mathbb{R},\mathcal{E})},
\quad w\in W^{1,2}_d(\mathbb{R},\mathcal{E}).
\end{eqnarray*}

 We define $$\mathcal{X}=\mathcal{E}^{a,b}_d(H,h)$$
as the Banach manifold consisting of all tuples
$$
w=((w_j)_{1\leq j\leq m},(t_j)_{1\leq j\leq m-1})\in \big(W^{1,p}_{loc}(\mathbb{R},\mathcal{E})\big)^m\times
({\mathbb{R}_+\setminus\{0\}})^{m-1}
$$
with the following properties:
\begin{description}
\item[(i)] For $1\leq j\leq m$ there exist $p_j,q_j\in{\rm Crit}^{[a,b]}(\mathcal{A}_H)$ such that
$$w_j(-\infty)=p_j,\quad w_j(+\infty)=q_j,\quad w_j\in \bar{w}_j+W^{1,2}_d(\mathbb{R},\mathcal{E}),$$
where  $\bar{w}_j:\mathbb{R}\to \mathcal{E}$ is a smooth map which for some $R\in\mathbb{R}$ satisfies $\bar{w}_j(t)\equiv p_j$ for $t\leq -R$ and $\bar{w}_j(t)\equiv q_j$ for $t\geq R$.
\item[(ii)] $p_{j+1}=\Psi_h(t_j)q_j$ for $1\leq j\leq m-1$.
\end{description}
There exist two natural evaluation maps
\begin{eqnarray*}
&&{\rm ev}_1:\mathcal{X}\to{\rm Crit}^{[a,b]}(\mathcal{A}_H),\quad w\mapsto w_1(-\infty)=p_1,\\
&&{\rm ev}_2:\mathcal{X}\to{\rm Crit}^{[a,b]}(\mathcal{A}_H),\quad w\mapsto w_m(+\infty)=q_m.
\end{eqnarray*}
The tangent space of $\mathcal{X}$ at $w$ can be identified with a subspace of
\begin{eqnarray*}
\begin{array}{c}
\bigoplus\limits_{j=1}^m\bigl(W^{1,2}_d(\mathbb{R},\mathcal{E})
\times T_{p_j}{\rm Crit}^{[a,b]}(\mathcal{A}_H)\times T_{q_j}{\rm Crit}^{[a,b]}(\mathcal{A}_H)\bigr)\times\mathbb{R}^{m-1}
\end{array}
\end{eqnarray*}
consisting of tuples $\omega=((\xi_j,\zeta_{j,1},\zeta_{j,2})_{1\leq j\leq m},(\tau_j)_{1\leq j\leq m-1})$
which satisfy
$$d\Psi_h(t_j)\zeta_{j,2}+\frac{d}{dt}(\Psi_h(0)q_j)\tau_j
=\zeta_{j+1,1}.$$
$T_w \mathcal{X}$ is a Banach space with norm
$$
\|\omega\|:=\sum\limits^m_{j=1}(\|\xi_j\|_{1,d}+|\zeta_{j,1}|_g
+|\zeta_{j,2}|_g)+\sum\limits^{m-1}_{j=1}|\tau_j|,
$$
where $|\cdot|_g$ is the norm with respect to the Riemannian metric $g$ on ${\rm Crit}(\mathcal{A}_H)$. Consider the map
$$\mathfrak{F}_K:\mathcal{X}\to \bigoplus\limits_{j=1}^m L^2_d(\mathbb{R},\mathcal{E}),\quad w\mapsto( w_j'(t)+\nabla^K \mathcal{A}_H(w_j(t)))_{1\leq j\leq m}.$$
Let $\mathcal{M}_K=\mathfrak{F}_K^{-1}(0)$. The differential of $\mathfrak{F}_K$ at $w\in\mathcal{M}_K$ is denoted by
$${\rm d}_w\mathfrak{F}_K:T_w\mathcal{X}\to \bigoplus\limits_{j=1}^m L^2_d(\mathbb{R},\mathcal{E}).$$

\noindent{\bf 5.2.2. Proof of Theorem~\ref{th:5.1}}. \quad We prove the theorem in three steps.\\
\textbf{Step 1.} \emph{Let $m=m(w)$ be the number of cascades. Then ${\rm d}_w\mathfrak{F}_K$ is a Fredholm operator of index}
\begin{equation}\label{eq:5.1.2}
{\rm ind}({\rm d}_w\mathfrak{F}_K)=i_{\rm rel}({\rm ev_1}(w))+{\rm dim}_{{\rm ev_1}(w)}{\rm Crit}^{[a,b]}(\mathcal{A}_H)-i_{\rm rel}({\rm ev_2}(w))+m-1.
\end{equation}

For $1\leq j\leq m$ denote by
$${\rm d}_{w,j}\mathfrak{F}_K:W^{1,2}_d(\mathbb{R},\mathcal{E})\to L^2_d(\mathbb{R},\mathcal{E})$$
the restriction of ${\rm d}_w\mathfrak{F}_K$ to
$$
W^{1,2}_d(\mathbb{R},\mathcal{E})\equiv
W^{1,2}_d(\mathbb{R},\mathcal{E})\times\{0\}\times\{0\}
\subset W^{1,2}_d(\mathbb{R},\mathcal{E})
\times T_{p_j}{\rm Crit}^{[a,b]}(\mathcal{A}_H)\times T_{q_j}{\rm Crit}^{[a,b]}(\mathcal{A}_H),
$$
 which is a Fredholm operator given by
\begin{equation*}
{\rm d}_{w,j}\mathfrak{F}_K=\frac{d}{dt}+ K(w_j(t)){\rm Hess}\mathcal{A}_H(w_j(t))
+\{dK(w_j(t))[\cdot]\}\nabla\mathcal{A}_H(w_j(t)).
\end{equation*}
By deforming $K(\cdot)$ linearly to zero as in Proposition~\ref{prop:4.2.1}, we see that
\begin{equation}\label{eq:5.1.3}
{\rm ind}({\rm d}_{w,j}\mathfrak{F}_K)={\rm ind}(L_{w,j}),
\end{equation}
where $L_{w,j}:W^{1,2}_d(\mathbb{R},\mathcal{E})\to L^2_d(\mathbb{R},\mathcal{E})$ is defined by
\begin{equation*}
L_{w,j}=\frac{d}{dt}+{\rm Hess}\mathcal{A}_H(w_j(t)).
\end{equation*}
 By conjugating $L_{w,j}$ with $s_d$
we define the operator
\begin{equation*}
\tilde{L}_{w,j}=s_d L_{w,j} s_{-d}:W^{1,2}(\mathbb{R},\mathcal{E})\to L^2(\mathbb{R},\mathcal{E}).
\end{equation*}
This is also Fredholm and satisfies
\begin{equation}\label{eq:5.1.4}
{\rm ind}(\tilde{L}_{w,j})={\rm ind}(L_{w,j}).
\end{equation}
For $\xi\in W^{1,2}(\mathbb{R},\mathcal{E})$ we calculate
\begin{eqnarray*}
\tilde{L}_{w,j}\xi&=&s_dL_{w,j}(s_{-d}\xi)\\
&=&\frac{d\xi}{dt}+\big({\rm Hess}\mathcal{A}_H(w_j(t))-d(\vartheta(t)+
\vartheta'(t)t){\rm id}\big)\xi.
\end{eqnarray*}
Set
$$
A_j(t):={\rm Hess}\mathcal{A}_H(w_j(t))-d(\vartheta(t)+
\vartheta'(t)t){\rm id}.
$$
Then
$$
A_j(-\infty)={\rm Hess}\mathcal{A}_H(p_j)+d\,{\rm id}\quad\hbox{and}\quad A_j(+\infty)={\rm Hess}\mathcal{A}_H(q_j)-d\,{\rm id}.
$$
The inequality $(\ref{eq:5.1.1})$ implies that $A_j(\pm\infty)$ are invertible and satisfy
\begin{eqnarray*}
&&V^-(A_j(-\infty))=V^-({\rm Hess}\mathcal{A}_H(p_j))\quad\hbox{and}\\
&&V^-(A_j(+\infty))=V^-({\rm Hess}\mathcal{A}_H(q_j))\oplus {\rm ker}({\rm Hess}\mathcal{A}_H(q_j)).
\end{eqnarray*}
It follows from Lemma~\ref{lem:4.2.2} that $\tilde{L}_{w,j}$ is a Fredholm operator of index
\begin{eqnarray}\label{eq:5.1.5}
{\rm ind}(\tilde{L}_{w,j})&=&{\rm dim}(V^-(A_j(-\infty)),V^-(A_j(+\infty)))\notag\\
&=&{\rm dim}(V^-(A_j(-\infty)),V^-({\rm Hess}\mathcal{A}_H(q_j)))\notag\\
&&+{\rm dim}(V^-({\rm Hess}\mathcal{A}_H(q_j)),V^-(A_j(+\infty)))\notag\\
&=&{\rm dim}(V^-({\rm Hess}\mathcal{A}_H(p_j)),V^-({\rm Hess}\mathcal{A}_H(q_j)))\notag\\
&&+{\rm dim}(V^-({\rm Hess}\mathcal{A}_H(q_j)),V^-({\rm Hess}\mathcal{A}_H(q_j))\oplus {\rm ker}({\rm Hess}\mathcal{A}_H(q_j)))\notag\\
&=& i_{\rm rel}(p_j)-i_{\rm rel}(q_j)-{\rm dim}_{q_j}{\rm Crit}^{[a,b]}(\mathcal{A}_H).
\end{eqnarray}
Combining the following index formulation
$${\rm ind}({\rm d}_w\mathfrak{F}_K)=\sum\limits_{j=1}^m{\rm ind}({\rm d}_{w,j}\mathfrak{F}_K)+{\rm dim}_{p_1}{\rm Crit}^{[a,b]}(\mathcal{A}_H)+\sum\limits_{j=1}^m{\rm dim}_{q_j}{\rm Crit}^{[a,b]}(\mathcal{A}_H)+m-1$$
with (\ref{eq:5.1.3}), (\ref{eq:5.1.4}) and (\ref{eq:5.1.5}), we obtain (\ref{eq:5.1.2}).\\
\textbf{Step 2.} \emph{For generic $K\in \mathbf{K}$ the set
$\mathcal{M}_K=\mathfrak{F}_K^{-1}(0)$ is a finite dimensional manifold with the local dimension}
${\rm dim}_w\mathcal{M}_K={\rm ind}\mathfrak{F}_K$.

We define
$$
\mathfrak{F}:\mathcal{X}\times \mathbf{K}\to \bigoplus\limits_{j=1}^m L^2_d(\mathbb{R},\mathcal{E}),
\quad (w,K)\mapsto  \mathfrak{F}_K(w),
 $$
and set
$$
\Theta:=\{(w,K)\in\mathcal{X}\times \mathbf{K}|\mathfrak{F}(w,K)=0\}.
$$
For each $(w,K)\in\Theta$,  the derivative of $\mathfrak{F}$ at $(w,K)$ is given by
$$
{\rm d}\mathfrak{F}(w,K)(\zeta,\kappa)={\rm d}\mathfrak{F}_K(w)\zeta
+(\kappa(w_j)\nabla \mathcal{A}_H(w_j))_{1\leq j\leq m},\quad\forall~ (\zeta,\kappa)\in T_{(w,K)}(\mathcal{X}\times \mathbf{K}).
$$
\textsf{We claim that ${\rm d}\mathfrak{F}(w,K)$ is surjective}. In fact,
since ${\rm d}\mathfrak{F}_K(w)$ is a Fredholm operator, it and hence  ${\rm d}\mathfrak{F}(w,K)$ has a closed range and a finite dimensional cokernel.
 So it suffices to prove that ${\rm Im}({\rm d}\mathfrak{F}(w,K))$ is dense.
 By a contradiction  we assume that $\bigl({\rm Im}({\rm d}\mathfrak{F}(w,K))\bigr)^{\perp}$ contains a
 nonzero element
$$
\eta=(\eta_j)_{1\leq j\leq m}\in \bigoplus\limits_{j=1}^mL^2_d(\mathbb{R},\mathcal{E}).
$$
Then we have
\begin{eqnarray}
&&\sum\limits_{j=1}^m\int_M\langle \eta_j, ({\rm d}\mathfrak{F}_K(w)\zeta)_j\rangle dt =0, \quad \forall~\zeta\in T_w\mathcal{X},\label{eq:5.1.6}\\
&&\sum\limits_{j=1}^m\int_M\langle \eta_j, \kappa(w_j)\nabla \mathcal{A}_H(w_j)\rangle dt =0, \quad \forall~\kappa\in \mathbf{K}.\label{eq:5.1.7}
\end{eqnarray}
(\ref{eq:5.1.6}) implies that for all $\xi_j\in W^{1,2}(\mathbb{R},\mathcal{E})$,
\begin{equation*}
\int_M\left\langle \eta_j,\frac{d}{dt}\xi_j(t)+ K(w_j(t)){\rm Hess}\mathcal{A}_H(w_j(t))\xi_j(t)
+\{dK(w_j(t))[\xi_j(t)]\}\nabla\mathcal{A}_H(w_j(t))\right\rangle dt=0.
\end{equation*}
It follows that $\eta_j$ is $C^1$  for $1\leq j\leq m$. These and (\ref{eq:5.1.7}) yield
that $\eta$ vanishes identically. (See Lemma~\ref{lem:6.2} for a similar proof). This
contradiction leads to the claim.  Hence it follows from the implicit function theorem that $\Theta$ is a Banach manifold.

Consider the projection
$$\pi:\Theta\to \mathbf{K},\quad (w,K)\to K,$$
which has the differential
$$
{\rm d}\pi(w,K):T_{(w,K)}\Theta\to T_K\mathbf{K},\quad (\zeta,\kappa)\to \kappa.
$$
The kernel of ${\rm d}\pi(w,K)$ is isomorphic to the kernel of ${\rm d}\mathfrak{F}_K(w)$. The fact that ${\rm d}\mathfrak{F}(w,K)$ is surjective implies that ${\rm d}\pi(w,K)$ has the same codimension as the image of ${\rm d}\mathfrak{F}_K(w)$. Thus ${\rm d}\pi(w,K)$ is a Fredholm operator of the same index as ${\rm d}\mathfrak{F}(w,K)$. It follows from the Sard-Smale Theorem that all regular values of $\pi$
forms a residual (and thus dense) subset in $\mathbf{K}$. And such regular values correspond to $K$ for which ${\rm d}\mathfrak{F}_K(w)$ is surjective for each $w\in\mathfrak{F}_K^{-1}(0)$.

\noindent{\bf Step 3.} By step 2 for generic $K\in\mathbf{K}$ the set $\widehat{\mathcal{M}}_m^{a,b}(x_0,x_1;H,K,h,g)$ can be endowed with the structure of a manifold with corners of dimension $\nu(x_0)-\nu(x_1)-1$. Set
$$
\mathcal{M}_{\leq l}(x_0,x_1):=\bigcup\limits_{1\leq m\leq l}\widehat{\mathcal{M}}_m^{a,b}(x_0,x_1;H,K,h,g),
\quad\forall l\in \mathbb{N}.
$$
 We show by induction on $l$ that for generic $K\in\mathbf{K}$ the set $\mathcal{M}_{\leq l}(x_0,x_1)$ has the structure of manifold of dimension $\nu(x_0)-\nu(x_1)-1$. For $l=1$ this is clear  since $\mathcal{M}_{\leq 1}(x_0,x_1)=\widehat{\mathcal{M}}_1^{a,b}(x_0,x_1;H,K,h,g)$. As we said above
  Theorem~\ref{th:5.1}, $\mathcal{M}_{\leq l}(x_0,x_1)$ can be compactified to a manifold with corners $\bar{\mathcal{M}}_{\leq l}(x_0,x_1)$ such that
$$
\partial \bar{\mathcal{M}}_{\leq l}(x_0,x_1)=\partial \widehat{\mathcal{M}}_{l+1}^{a,b}(x_0,x_1;H,K,h,g).
$$
So $\mathcal{M}_{\leq l+1}(x_0,x_1)=
\mathcal{M}_{\leq l}(x_0,x_1)\cup\widehat{\mathcal{M}}_{l+1}^{a,b}(x_0,x_1;H,K,h,g)$ has a finite dimensional manifold structure with
$$
{\rm dim}\mathcal{M}_{\leq l+1}(x_0,x_1)={\rm dim}\mathcal{M}_{\leq l}(x_0,x_1)=\nu(x_0)-\nu(x_1)-1.
$$
\qed

Assume that $\nu(x_0)=\nu(x_1)+1$. Then  the set $\widehat{\mathcal{M}}^{a,b}(x_0,x_1;H,K,h,g)$ is finite
by Theorem~\ref{th:5.1}. Put
$$n(x_0,x_1)=\#\widehat{\mathcal{M}}^{a,b}(x_0,x_1;H,K,h,g).$$
The chain group $BC_*^{[a,b]}(H,K,h,g)$ is defined as the finite dimension $\mathbb{Z}_2$-vector space given by
$$
BC_*^{[a,b]}(H,K,h,g):={\rm Crit}^{[a,b]}(h)\otimes\mathbb{Z}_2,
$$
where  ${\rm Crit}^{[a,b]}(h):={\rm Crit}(h)\cap {\rm Crit}^{[a,b]}(\mathcal{A}_H)$.
The grading is given by the above index $\nu$ and the differential operator is defined by
$$
\partial x=\sum\limits_{\substack{y\in{\rm Crit}^{[a,b]}(h),\\\nu(y)=\nu(x)-1}}
\big(n(x,y)\hspace{1mm}{\rm mod}[2]\big)y.$$
Using the compactness in Proposition~\ref{prop:3.4} and a standard gluing construction as in \cite{Fra}, we can prove that $\partial^2=0$.
Thus $\big(BC_*^{[a,b]}(H,K,h,g),\partial\big)$ is a chain complex, still called the
\emph{Rabinowitz Floer complex} of $\mathcal{A}_H$.
 The corresponding homology
$$
HF_*^{[a,b]}(H,K,h,g):=H_*\big(BC_*^{[a,b]}(H,K,h,g)
\partial\big)
$$
is called the \emph{Rabinowitz Floer homology} of $\mathcal{A}_H$.
Standard arguments show that $HF_*^{[a,b]}(H,K,h,g)$ is independent up to canonical isomorphism of the choices of $H,K,h$ and $g$,  see \cite[Frauenfelder]{Fra} for details. So
$HF_*^{[a,b]}(H,K,h,g)$ will be simply denoted by $HF_*^{[a,b]}(H)$.

\subsection{Continuation of the Rabinowitz-Floer homology}\label{sec:5.2}
In this subsection, we show that under a small perturbation of the pair $(H,K)$ there exists a natural isomorphism between the Rabinowitz-Floer homology of $\mathcal{A}_H$ and that of the pertubed functional $\mathcal{A}_{\widetilde{H}}$. Then by taking a partition of a smooth path connecting from $(H_0,K_0)$ to $(H_1,K_1)$, under suitable hypotheses we prove $RHF_k^{[a,b]}(H_0,K_0)=RHF_k^{[a,b]}(H_1,K_1)$.
In order to control length of this paper we only consider the Morse situation.
The similar proof can be completed in the Morse-Bott situation; see \cite[Appendix A]{Fra}.

Fix a constant $\epsilon$ as in Proposition~\ref{prop:3.3}. Assume that $(H_0,K_0), (H_1,K_1)\in\Omega_{reg}$ and
\begin{equation}\label{5.2.1}
\sup_{z\in E_s}\int_M|H_0(x,z)-H_1(x,z)|dx+\|K_0-K_1\|_{\mathbf{K}}<\frac{\epsilon}{10}.
\end{equation}
Let $\beta(t)\in C^{\infty}(\mathbb{R},[0,1])$ satisfy $\beta(t)\equiv0$ for $t\leq0$, $\beta(t)\equiv1$ for $t\geq1$, and $0\leq\beta^\prime(t)\leq2$ for all $t$. We define the $t$-dependent functions $H(t,x,z)$, $K(t,w)$ by
\begin{equation}\notag
H(t,x,z)=(1-\beta(t))H_0(x,z)+\beta(t)H_1(x,z),\quad K(t,w)=(1-\beta(t))K_0(w)+\beta(t)K_1(w).
\end{equation}
It is not hard to check that they satisfy the assumptions (i)-(iii)
 with $A=\epsilon/5$ at the beginning of Section~3.3.
 By replacing $(H,K)$ by an arbitrary small perturbation we can assume that $(H,K)$ is regular in the sense that
  the map $\mathcal{F}_{H,K}:Q_1\to Q_0$ given by
$$
\mathcal{F}_{H,K}(w)(t)=\frac{dw(t)}{dt}+(I+K(t,w(t)))\nabla \mathcal{A}_H(t,x,w(t))
$$
 is transversal to $0\in Q_1$,  where $Q_0$ and $Q_1$ are as in (\ref{eq:4.2.2}).
Hence for given any pair of critical points $w_0\in{\rm Crit}(\mathcal{A}_{H_0})$ and $w_1\in{\rm Crit}(\mathcal{A}_{H_1})$,
if $w(t)$ is a solution of (\ref{eq:3.3.2}) satisfying $w(-\infty)=w_0$ and $w(+\infty)=w_1$,
 by Proposition~\ref{prop:3.3} $w(t)$ is uniformly bounded by a constant depending only on $w_0$ and $w_1$.
This uniform boundedness implies precompactness,
therefore we can define the moduli space of trajectories of the negative non-autonomous gradient flow
\begin{eqnarray} \notag
\overline{\mathcal{M}}(w_0,w_1)=\left\{w\in C^1(\mathbb{R},\mathcal{E})\bigg|\begin{array}{l}
w(t)\hbox{ is a solution of (\ref{eq:3.3.2}) with}\\  w(-\infty)=w_0 \hbox{ and } w(+\infty)=w_1
\end{array} \right\}.
\end{eqnarray}
One can show that $\overline{\mathcal{M}}(w_0,w_1)$ is either empty or a  manifold of dimension
$$
{\rm dim}\overline{\mathcal{M}}(w_0,w_1)=i_{\rm rel}^{H_0}(w_0)-i_{\rm rel}^{H_1}(w_1),
$$
where $i_{\rm rel}^{H_j}(w_j)$ is the relative index with respect to $H_j$ at $w_j$, $j=0,1$.
The key techniques of the proof are compactication of broken trajectories and the gluing construction very similar to that of the autonomous case, which we will not reproduce here. If $i_{\rm rel}^{H_0}(w_0)=i_{\rm rel}^{H_1}(w_1)$, then the integer $n(w_0,w_1):=\#\overline{\mathcal{M}}(w_0,w_1)$ is finite. For each $k\in\mathbb{Z}$, we consider a homomorphism
$$
\Psi_{01}:{\rm CF}_k(H_0)\to{\rm CF}_k(H_1)
$$
defined by
$$
\Psi_{01}\left(\sum^l_{j=1}m_jx_j\right)=\sum^l_{j=1}\sum\limits_{y\in{\rm Crit}_k(\mathcal{A}_{H_1})}(n(x_j,y)\text{ mod }[2])m_jy
$$
for $\sum^l_{j=1}m_jx_j\in {\rm CF}_k(H_0)$ with $x_j\in{\rm Crit}(\mathcal{A}_{H_0})$ and $m_j\in\mathbb{Z}_2$, $j=1,\cdots,m$.
We claim: \textsf{ $\Psi_{01}$ is a chain homomorphism}. That is, the following diagram communicates
\[
\begin{CD}
{\rm CF}_{k+1}(H_0) @>{\partial_{k+1}^0}>>{\rm CF}_k(H_0) \\
@VV\Psi_{01}V   @VV\Psi_{01}V\\
{\rm CF}_{k+1}(H_1)@>{\partial_k^1}>> {\rm CF}_k(H_1),
\end{CD}
\]
where $\partial^j$, $j=0,1$, are boundary operators corresponding to $(H_j,K_j)$.
The proof is standard. In fact, we only need to consider the 1-dimension moduli space $\overline{\mathcal{M}}(w_0,w_1)$ with $i_{\rm rel}^{H_0}(w_0)=i_{\rm rel}^{H_1}(w_1)+1$. The boundary of $\overline{\mathcal{M}}(w_0,w_1)$ then splits into two parts:
\begin{eqnarray*}
\partial\overline{\mathcal{M}}(w_0,w_1)&=&\left(\bigcup\limits_{x\in{\rm Crit}_{k+1}(\mathcal{A}_{H_1})}\overline{\mathcal{M}}(w_0,x)
\times\mathcal{M}_{H_1,K_1}(x,w_1)\right)\\
&&\bigcup\left(\bigcup\limits_{y\in{\rm Crit}_k(\mathcal{A}_{H_0})}\mathcal{M}_{H_0,K_0}(w_0,y)
\times\overline{\mathcal{M}}(y,w_1)\right).
\end{eqnarray*}
The first part appears in $\partial^0\circ\Psi_{01}$ while the second parts does in $\Psi_{01}\circ\partial^1$. The desired claim follows.

 Moreover, we have
\begin{lemma}\label{lem:5.1}
If $(H_l,K_l)\in\Omega_{reg}$, $l=1,2,3$, satisfy
 $$
 \sup\limits_{z\in E_s}\int_M|H_m(x,z)-H_n(x,z)|dx+\|K_0-K_1\|_{\mathbf{K}}<\frac{\epsilon}{5}\quad
 \forall m,n\in\{1,2,3\},
 $$
  then $\Psi_{ln}=\Psi_{lm}\circ\Psi_{mn}$ and $\Psi_{ll}=id$.
In particular, $\Psi_{12}$ is an isomorphism.
\end{lemma}

Again using the result of compactness and the standard arguments in \cite{AnV,CiF,SaD}, we arrive at the conclusion of the above lemma. The proof is omitted here. By transversality, for each $H$ satisfying $({\bf H}1)-({\bf H}4)$, there exists  $(\tilde{H},\tilde{K})\in\Omega_{reg}$ such that $|H-\widetilde{H}|$ is small enough, then one can define the Rabinowitz-Floer homology of $\mathcal{A}_H$ to be that of $\mathcal{A}_{\widetilde{H}}$.
\begin{proposition}[\textbf{Global continuation}]\label{prop:5.1}
If $(H_0,K_0),(H_1,K_1)\in\Omega_{reg}$, then it holds that
$$RHF_*^{[a,b]}(H_0,K_0)=RHF_*^{[a,b]}(H_1,K_1).$$
\end{proposition}
\begin{proof}
We prove the propostion in two steps.\\
\textbf{Step 1.} Let us make an additional hypothesis that $\sup_{z\in E_s}\int_M|H_0(x,z)-H_1(x,z)|dx$ is finite. Let $\epsilon$ be as in (\ref{5.2.1}). Given a smooth path $(H_s,K_s)$ with $s\in[0,1]$ from $(H_0,K_0)$ to $(H_1,K_1)$,  for example one can choose $H_s(x,z)=sH_0(x,z)+(1-s)H_1(x,z)$ and then take a partition $0=s_0< s_1<\cdots<s_m=1$ such that
$$
\sup\limits_{z\in E_s}\int_M|H_{s_{j+1}}(x,z)-H_{s_j}(x,z)|dx<\frac{\epsilon}{5},\quad
\|K_{s_{j+1}}-K_{s_j}\|_{\mathbf{K}}<\frac{\epsilon}{5},\quad j\in\{0,\ldots, m\}.
$$
It follows from Lemma~\ref{lem:5.1} that there exist isomorphisms
$$
\Psi_{j,j+1}:RHF_*(H_{s_j},K_{s_j})=RHF_*(H_{s_{j+1}},K_{s_{j+1}}).
$$
By composing these isomorphisms we get an isomorphism between the Rabinowitz-Floer homologies of $\mathcal{A}_{H_0}$ and $\mathcal{A}_{H_1}$.

\noindent{\bf Step 2.} If $(z,\lambda)\in{\rm Crit}_k^{[a,b]}(\mathcal{A}_{H_l})$ then $z\in \Sigma_1(H_l)=\{z\in E_s\,|\,\int_MH(x, z(x))dx\le 1\}$,
 $l=0,1$. From the proof of Proposition~\ref{prop:3.1}, we see that the assumption $({\bf H}4)$ implies that $\Sigma_1(H_0)$ and $\Sigma_1(H_1)$ are bounded sets in $E_s$. It follows from Proposition~\ref{prop:3.2} that the negative gradient flow lines connecting two critical points are uniformly bounded in $\mathcal{E}$. Take a ball $B_R(0)\subset E_s$ such that $\Sigma_1(H_l)$ and the $z$-components of corresponding connecting orbits are all contained within it. For $\delta>0$, one chooses a smooth function $\chi_\delta(t)$ such that $\chi_\delta(t)=1$ for $0\leq t\leq R$ and $\chi_\delta(t)=0$ for $t\geq (R+\delta)$. Given $K^\delta\in \mathbf{K}$, consider the modified function on $\Sigma M\oplus\Sigma M$,
 $$
H^\delta(x,z)=\chi_\delta(\|z\|)H_0(x,z)+(1-\chi_\delta(\|z\|))H_1(x,z).
$$
Clearly,  $\sup_{z\in E_s}\int_M|H^\delta(x,z)-H_1(x,z)|dx<+\infty$.
By step 1, it holds
\begin{equation}\label{eq:5.2}
RHF_*^{[a,b]}(H^\delta,K^\delta)=RHF_*^{[a,b]}(H_1,K_1)
\end{equation}
 Moreover, since $\Sigma_1(H^\delta)\to\Sigma_1(H_0)$ as $\delta\to 0$, and the Rabinowitz-Floer groups are defined in terms of the critical points and connecting orbits between them,  we obtain
\begin{equation}\label{eq:5.3}
RHF_*^{[a,b]}(H^\delta,K^\delta)=RHF_*^{[a,b]}(H_0,K_0).
\end{equation}
The desired result follows from (\ref{eq:5.2}) and (\ref{eq:5.3}).
\end{proof}

Because of the above theorem we can simply write $RHF_k^{[a,b]}(H)=RHF_k^{[a,b]}(H,K)$.
As we said before, the global continuation also holds in the Morse-Bott situation.
In particular, if $\mathcal{A}_H$ is Morse, we can take  $h$ vanishing identically
on ${\rm Crit}^{[a,b]}(\mathcal{A}_H)$, and obtain $RHF_*^{[a,b]}(H)=HF_*^{[a,b]}(H)$.

\section{Transversality}\label{sec:6}
In this section we first show that
 the nonlinearity $H$ can be slightly perturbed so that $\mathcal{A}_{H}$ is Morse. Then
  following the ideas of Abbondandolo and Majer \cite{AbM} we can make a small perturbation of $\mathcal{A}_{H}$ such that the perturbed functional $\widetilde{\mathcal{A}}$ satisfies the Morse-Smale condition.

Consider the Gevrey space $\mathbf{G}$ of $C^\infty$ functions $h:\Sigma M\oplus\Sigma M\to\mathbb{R}$ with the norm
$$|h|_\mathbf{G}:=\sup\limits_{k\in\mathbb{N}}
\frac{\|h\|_{C^k}}{(k!)^4}<+\infty.$$
Here the $\|\cdot\|_{C^k}$-norm can be explicitly given  as follows:
Because of compactness of $M$ we choose a finite open cover $\{\mathcal{U}_k\}^m_{k=1}$ of $M$ consisting
of domains of chart maps $\varphi_k:\mathcal{U}_k\to\Omega_k$, $k=1,\cdots,m$,  where $\Omega_k=\varphi_k(\mathcal{U}_k)\subset\mathbb{R}^n$
is an open unit ball; and therefore there exist  local trivializations $\Phi_k:\Sigma M\oplus\Sigma M|_{\mathcal{U}_k}\to\mathcal{U}_k\times\Sigma^n\times\Sigma^n$, $k=1,\cdots,m$.
Fix a partition
$\{\lambda_k\}^m_{k=1}$ of unity subordinate to  $\{\mathcal{U}_k\}^m_{k=1}$, and define for any smooth function $h:\Sigma M\oplus\Sigma M\to\mathbb{R}$,
$$
\|h\|_{C^k}=\sum\limits_k\sup\limits_{\bar{\Omega}_k\times
(\Sigma^n)^2}\left\|d^k\bigl((\lambda_k\circ\varphi_k^{-1})\cdot(h\circ\Phi_k^{-1})\bigr)(x,u,v)
\right\|.
$$
It is not hard to see that any alternative choice of finite open covering of charts, trivializations and partition of unity gives an equivalent
norm on $\mathbf{G}$. Let us fix such choice. Then $\mathbf{G}$ is a separable Banach space \cite{Hom}. 

\begin{lemma}\label{lem:6.1}
Assume that $H\in C^2(\Sigma M\oplus\Sigma M)$ satisfies $({\bf H}1)-({\bf H}4)$. Then
$\mathcal{A}_{H+h}$ is a Morse function for a generic perturbation $h$ of $H$ in $\mathbf{G}$.
\end{lemma}
\begin{proof}
We define a map $\Psi:\mathcal{E}\times\mathbf{G}\to \mathcal{E}$ by
\begin{equation}\label{eq:6.1}
\Psi(w,h)=\nabla\mathcal{A}_{H+h}(w),
\end{equation}
where $w=(z,\lambda)\in \mathcal{E}$. One can easily checks that $\Psi$ is a map of class $C^1$.

We first prove that $0$ is a regular value of $\Psi$.
Since we have assume $0\notin{\rm spec}(D)$,
for each $(w,h)\in \Psi^{-1}(0)$ with $w=(z,\lambda)$ it must hold that $\lambda\ne 0$.
 It is not hard to see that
 the derivative of $\Psi$ at $(w,h)\in \Psi^{-1}(0)$  with respect to $w$ is given by
\begin{equation}\notag
\begin{array}{l}
{\rm d}_w\Psi(w,h)=
\begin{pmatrix}
\mathcal{D}_sL&0 \\
0&1
\end{pmatrix}
+
\begin{pmatrix}
-\lambda\mathcal{D}_s(H_{zz}+h_{zz})&-\mathcal{D}_s(H_z+h_z) \\
-\big(\mathcal{D}_s(H_z+h_z)\big)^*&-1
\end{pmatrix}.
\end{array}
\end{equation}
This is Fredholm operator with Fredholm index~$0$ since it is a compact perturbation of the invertible operator $\scriptscriptstyle\begin{pmatrix}\begin{smallmatrix}\mathcal{D}_sL&0 \\0&1\end{smallmatrix}\end{pmatrix}$. Hence the range of ${\rm d}_w\Psi(w,h)$ has finite codimension. To prove that ${\rm d}\Psi(w,h)$ is surjective, we only need to show that the range of the derivative of $\Psi$ with respect to $h$ at $(w,h)$ given by
\begin{equation}\label{eq:6.2}
\begin{array}{l}
{\rm d}_h\Psi(w,h)X=
\begin{pmatrix}
-\lambda\mathcal{D}_s X_z(x,z) \\
 -\int_M X(x,z)dx
 \end{pmatrix}
\end{array}
\end{equation}
is dense in $\mathcal{E}$, where $X\in T_h\mathbf{G}=\mathbf{G}$. Choosing $X\equiv constant\neq0$ and substituting it into (\ref{eq:6.2}), we see that the second component of ${\rm d}_h\Psi(w,h)X$ spans $\mathbb{R}$.

 \noindent{\bf Claim}. {\it The first component of ${\rm d}_h\Psi(w,h)$ has dense range in $E_s$}.

  In fact, consider the element of $C^1(M,\Sigma M\oplus\Sigma M)$
 of form $(a\phi, b\varphi)$, where $(\phi, \varphi)\in C^1(M,\Sigma M\oplus\Sigma M)$
 and $a, b\in C^\infty(M)$ satisfy the following condition
 $$
 \sup_{x\in M}\{|{\rm d}^ja(x)|,|{\rm d}^jb(x)|\}\leq C(j!)^4\quad\forall j\in\mathbb{N}\cup\{0\}
 $$
 for some constant $C$. Define the function $\bar{X}:\Sigma M\oplus\Sigma M\to\mathbb{R}$ by
\begin{equation*}
\bar{X}(x,u,v)=a(x)\langle\phi(x),u\rangle+b(x)\langle\varphi(x),v\rangle.
\end{equation*}
 Then $\bar{X}\in \mathbf{G}$ and
 $$
 \bar{X}_z(x,z(x))=\big( a(x)\phi(x),b(x)\varphi(x)\big)^T.
 $$
 Denote by $\triangle$ the set consisting of all such $\bar{X}$.
Since  $\{\bar{X}_z\,|\, \bar{X}\in\triangle\}$ is dense in $C^1(M,\Sigma M)\times C^1(M,\Sigma M)$, and $\mathcal{D}_s$ maps this set into a dense subspace in $E_s$, we deduce that the set
$$
\{\hbox{the first component of}\;{\rm d}_h\Psi(w,h)\bar{X}\,|\,\bar{X}\in\triangle\}
=\{-\lambda\mathcal{D}_s X_z\,|\,\bar{X}_z\in \triangle\}
$$
is dense in $E_s$. Here we use the fact that $\lambda\ne 0$, which comes from
the assumption that $0\notin{\rm Spect}(D)$ as showed at the beginning.
 Hence $0$ is a regular value of $\Psi$.

Next we consider the $C^1$-submanifold $\mathcal{Z}=\{(w,h)\in\mathcal{E}\times\mathbf{G}\,|\,
\Psi(w,h)=0\}$ and the projection $\pi:\mathcal{Z}\to\mathbf{G}$ given by $\pi(w,h)=h$. By a standard argument \cite{ScM}, the Fredholm property of ${\rm d}_w\Psi(w,h)$ implies that the derivative ${\rm d}\pi(w,h)$ is Fredholm and has the same index $0$ as ${\rm d}_w\Psi(w,h)$.
Therefore all regular values of $\pi$ form a residual subset of $\mathbf{G}$
by the Sard-Smale theorem \cite{Sma}. Moreover,  $h\in \mathbf{G}$ is a regular value of $\pi$
if and only if  $\mathcal{A}_{H+h}$ is a Morse functional. Hence $\mathcal{A}_{H+h}$ is a Morse functional for a generic $h\in\mathbf{G}$.
\end{proof}

 With the same strategy as the above arguments we shall prove
\begin{lemma}\label{lem:6.2}
Let $H\in C^2(\Sigma M\oplus\Sigma M)$ satisfying $({\bf H}1)-({\bf H}4)$. Then for generic $K\in\mathbf{K}$ the map $\mathcal{F}_{H,K}:Q_1\to Q_0$ in (\ref{eq:4.2.2}) satisfying $i_{\rm rel}(w_1)-i_{\rm rel}(w_2)=0$
has the regular value $0\in Q_0$.
\end{lemma}

\begin{proof}
We divide our proof into two steps.

\noindent{\bf Step 1.} Consider the map $\mathcal{F}_H:Q_1\times \mathbf{K}\to Q_0$ defined by
$$
\mathcal{F}_H(w,K)=\mathcal{F}_{H,K}(w)=\frac{d w}{dt}+ \nabla^K\mathcal{A}_H(w).
$$
It is of class~$C^1$ by our assumption, and the derivative of $\mathcal{F}_H$ at $(w,K)$ is given by
$$
{\rm d}\mathcal{F}_H(w,K)(y,\kappa)={\rm d}\mathcal{F}_{H,K}(w)y+\kappa(w)\nabla \mathcal{A}_H(w),\quad\forall~ (y,\kappa)\in T_{(w,K)}Q_1\times \mathbf{K}.
$$
Since ${\rm d}\mathcal{F}_{H,K}(w)$ is a Fredholm operator with the index $i_{\rm rel}(w_1)-i_{\rm rel}(w_2)=0$, it has a closed range
and a finite dimensional cokernel. Therefore, ${\rm d}\mathcal{F}_H(w,K)$ has a closed range and a finite dimensional cokernel. We claim that \textsf{$0\in Q_0$ is a regular value of $\mathcal{F}_H$}. Arguing by contradiction, assume that there exists $(w,K)\in(\mathcal{F}_H)^{-1}(0)$ such that ${\rm d}\mathcal{F}_H(w,K)$ is not surjective.
Then there exists $\psi\in L^2(\mathbb{R},\mathcal{E})\setminus\{0\}$ such that
\begin{eqnarray}
&&\int_M\langle\psi(t),{\rm d}\mathcal{F}_{H,K}(w(t))y\rangle dt=0, \quad\forall~y\in W^{1,2}(\mathbb{R},\mathcal{E}),\label{eq:6.3}\\
&&\int_M\langle\psi(t),\kappa(w(t))\nabla \mathcal{A}_H(w(t))\rangle dt=0, \quad\forall~\kappa\in \mathbf{K}.\label{eq:6.4}
\end{eqnarray}
(\ref{eq:6.3}) implies that $\psi(t)$ is a weak solution of the adjoint equation $({\rm d}\mathcal{F}_{H,K}(w(t)))^*\psi(t)=0$ and thus continuous. For any $\kappa\in\mathcal{NS}(\mathcal{E},C^2\oplus\mathbb{R})$ and a fixed $t_0\in\mathbb{R}$, we put
\begin{equation}\label{eq:6.5}
\kappa^\epsilon(w)=\frac{1}{\epsilon}\rho\bigg(\frac{\|w-w(t_0)\|_\mathcal{E}}
{\epsilon}\bigg)\kappa,
\end{equation}
where $\rho$ is as in (\ref{met:4}). Substituting $\kappa^\epsilon$ into (\ref{eq:6.4}) and taking $\epsilon\to 0$, we get
\begin{equation}\label{eq:6.6}
\frac{\int^1_{-1}e^{-1/(1-s^2)}ds}{\|w'(t_0)\|_\mathcal{E}}
\langle\psi(t_0),\kappa\nabla \mathcal{A}_H(w(t_0))\rangle=0.
\end{equation}
Since $\nabla \mathcal{A}_H(w(t_0))\neq0$ and $C^2\oplus\mathbb{R}$ is dense in $\mathcal{E}$, it follows from Proposition~\ref{prop:2.2} that $\psi(t_0)=0$. But $t_0$ is arbitrary.
 We arrive at  $\psi=0$, which contradicts with our assumption $\psi\neq0$. So ${\rm d}\mathcal{F}_H(w,K)$ is onto.

\noindent{\bf Step 2.} Now $\mathcal{Z}:=\mathcal{F}_H^{-1}(0)$ is a $C^1$-submanifold of $Q_1\times \mathbf{K}$ by Step 1,
and the projection $Q_1\times \mathbf{K}\to \mathbf{K}$ restricts to a Fredholm map $\pi:\mathcal{Z}\to\mathbf{K}$ with  index $0$.
 Hence the set of $\pi$'s regular values is of second category by the Sard-Smale theorem and satisfies the property of Lemma~\ref{lem:6.2}.
\end{proof}

\begin{remark}\label{rem:1}
{\rm Lemma~\ref{lem:6.1} and Lemma~\ref{lem:6.2} imply that for generic $h\in\mathbf{G}$ and generic $K\in\mathbf{K}$ the functional $\mathcal{A}_{H+h}$ has the \emph{Morse-Smale property up to order $0$}, that is to say, for each two critical points $w_1,w_2$ of $\mathcal{A}_{H+h}$
such that $i_{\rm rel}(w_1)-i_{\rm rel}(w_2)=0$
the unstable manifold of $w_1$ and the stable mainfold of $w_2$  meet transversally. However,
according to the usual method the perturbed functional should  own the Morse-Smale property up to at least order $2$ for constructing the Rabinowitz-Floer complex.
 This requires that $\nabla\mathcal{A}_{H}$ should be at least of class~$C^3$ because using the Sard-Smale theorem for Fredholm maps to obtain the Morse-Smale property requires the regularity to be strictly higher than the Fredholm index. But it is regrettable that when $\dim M\ge 3$ the functional $\mathcal{A}_H:\mathcal{E}\to \mathbb{R}$
 is at most of class $C^2$  even if $H$ is $C^\infty$.
Fortunately,  the functional $\mathcal{A}_H:\mathcal{E}\to\mathbb{R}$
can be written as
\begin{equation} \notag
\mathcal{A}_H(z,\lambda)=\frac{1}{2}\int_M\langle Lz(x),z(x)\rangle dx+G(z,\lambda),
\end{equation}
where $G(z,\lambda):=-\lambda\int_M\{H(x,z(x))-1\}dx$ and
  the gradient of the functional $G:\mathcal{E}\to\mathbb{R}$,
\begin{equation}\notag
(z,\lambda)\mapsto\nabla G(z,\lambda)=\begin{pmatrix}
-\lambda\mathcal{D}_s H_z(x,z) \\
 -\int_M\big(H(x,z)-1\big)dx
\end{pmatrix},
\end{equation}
is a compact map on $\mathcal{E}$  by Proposition~\ref{prop:2.1}. Thus according to \cite[Appendix B]{AbM}
(see \cite[page 758]{AbM} for precise explanation)  we can approximate $\mathcal{A}_H$ by a smooth functional $\widetilde{\mathcal{A}}$ on $\mathcal{E}$
 with  the Morse-Smale property up to every order.
}
\end{remark}
\begin{lemma}\label{lem:6.3}
Given $h$ as in Lemma~\ref{lem:6.1} and two real numbers $a,b$ with $a<b$, for each $\epsilon>0$ there exists a functional $\widetilde{\mathcal{A}}\in C^\infty(\mathcal{E})$ such that
\end{lemma}
\begin{description}
\item[(i)] $\|\mathcal{A}_{H+h}-\widetilde{\mathcal{A}}\|_{C^2(\mathcal{E})}<\epsilon$,
\item[(ii)] $\widetilde{\mathcal{A}}$ satisfies the $(PS)_c$-condition in $[a-\epsilon,b+\epsilon]$,
\item[(iii)] $\widetilde{\mathcal{A}}$ has the Morse-Smale property up to each order,
\item[(iv)] $\widetilde{\mathcal{A}}$ and $\mathcal{A}_{H+h}$ have the same critical points and the same connecting orbits.
\end{description}

Remark that by Lemma~\ref{lem:6.3} the perturbed functional $\widetilde{\mathcal{A}}$ can be used to define a boundary homomorphism
$$\partial: C_k(\widetilde{\mathcal{A}},[a,b])\longrightarrow C_{k-1}(\widetilde{\mathcal{A}},[a,b])$$
which is the same as~(\ref{e:BH}). The invariance of the homology \cite[Section 9]{AbM} implies that different perturbed functionals give the same homology.
The Rabinowitz-Floer homology $RHF_\ast(H)$ is then defined to be the homology of the above complex.


\section{Existence results for the coupled Dirac system}\label{sec:7}
\setcounter{equation}{0}
In this section, we prove Theorem~\ref{th:1.1} via computing the Rabinowitz-Floer homology. To do it, we choose a special nonlinearity
$H_0(x,u,v)=\frac{1}{2}(|u|^2+|v|^2)$ whose homology can be easily computed, then by continuation we get the desired result.

\subsection{Computations of Rabinowitz-Floer homology}\label{sec:7.1}

By (\ref{eq:2.6}), each critical $(z,\lambda)$ of $\mathcal{A}_{H_0}$ satisfies
\begin{eqnarray}\label{eq:7.1}
\left\{ \begin{array}{l}
Lz=\lambda z,\\
\textstyle\int_M H_0(x,z) dx=1.
\end{array} \right.
\end{eqnarray}
So each connected component $\sigma_k$ of ${\rm Crit}(\mathcal{A}_{H_0})$ has the form
$$
\{(z_k,\bar{\lambda}_k)\in\mathcal{E}\times\mathbb{R}\,|\,(z_k,\bar{\lambda}_k)\,\hbox{satisfies}\;(\ref{eq:7.1})\},
$$
where  $\bar{\lambda}_k$ is a fixed eigenvalues of $L$. If $\bar{\lambda}_k$ has multiplicity $m_k$, then $\sigma_k$
is a manifold diffeomorphic to a sphere $S^{2m_k-1}$ and has the tangent space at $(z_k,\bar{\lambda}_k)$
\begin{equation}\notag
T_{(z_k,\bar{\lambda}_k)}\sigma_k=\{(z,0)\in \mathcal{E}\times\mathbb{R}\,|\,
Lz=\bar{\lambda}_kz,\hspace{2mm}(z_k,z)_{L^2}=0\}.
\end{equation}
Since
\begin{equation}\notag
{\rm Hess}\mathcal{A}_{H_0}(z,\lambda)=\begin{pmatrix}
\mathcal{D}_s L-\lambda\mathcal{D}_s&-\mathcal{D}_sz \\
-\big(\mathcal{D}_sz\big)^*&0
\end{pmatrix},
\end{equation}
it is easy to check that
$$
{\rm Ker}\big({\rm Hess}\mathcal{A}_{H_0}(z_k,\bar{\lambda}_k)\big)=T_{(z_k,\bar{\lambda}_k)}\sigma_k.
$$
Therefore, $\mathcal{A}_{H_0}$ is a Morse-Bott functional. As done in the construction of the homology in subsection~\ref{sec:5.1}, for each $k\in\mathbb{Z}$ we  can choose a Morse function $h_k$ and a Riemannian metric $g_k$ on $\sigma_k$
such that $h_k$ has precisely a maximum point $p_k^+$ and a minimum point $p_k^-$. It follows that
$$
\nu(p_k^+)=i_{\rm rel}(p_k^+)+2m_k-1\quad\hbox{and}\quad\nu(p_k^-)=i_{\rm rel}(p_k^-)
$$
since ${\rm ind}(p_k^+)=2m_k-1$ and ${\rm ind}(p_k^-)=0$. Let $h$ be the Morse function on the critical manifold of $\mathcal{A}_{H_0}$, which coincides with $h_k$ on each connected component $\sigma_k$.
To compute the relative indices of critical points of $\mathcal{A}_{H_0}$, we need

\begin{lemma}[{\cite[Proposition 2.4]{AbM}}]\label{lem:7.1}
Let $H_1, H_2$ be two Hilbert spaces, and $T:H_1\to H_2$ an injective bounded linear operator. If $U,V$ are commensurable subspaces of $H_1$, then $T(U),T(V)$ are commensurable subspaces of $H_2$ and ${\rm dim}(T(U),T(V))={\rm dim}(U,V)$.
\end{lemma}

  Denote by
  $$
  A=\mathcal{D}_sL\oplus Id_{\mathbb{R}},\hspace{4mm}B(z,\lambda)={\rm Hess}\mathcal{A}_H(z,\lambda).
  $$
Consider unbounded self-adjoint operators on $L^2(M,\Sigma M)\times L^2(M,\Sigma M)\times\mathbb{R}$ defined by
\begin{equation*}
C=\begin{pmatrix}
 L&0 \\
0&1
\end{pmatrix},\hspace{4mm}D(z,\lambda)=\begin{pmatrix}
 L-\lambda Id&-z \\
-(z)^*&0
\end{pmatrix}.
\end{equation*}
Let $j:\mathcal{E}\hookrightarrow L^2(M,\Sigma M)\times L^2(M,\Sigma M)\times\mathbb{R}$ be the inclusion map. Then (\ref{eq:2.4}) implies that
\begin{eqnarray}
&&(Aw,w)_{\mathcal{E}}=\big(C(j(w)),j(w)\big)_{L^2\times\mathbb{R}}\quad\forall w\in \mathcal{E},\label{eq:7.2}\\
&&(B(z,\lambda)w,w)_{\mathcal{E}}=\big(D(z,\lambda)(j(w)),
j(w)\big)_{L^2\times\mathbb{R}}\quad\forall w\in \mathcal{E}.\label{eq:7.3}
\end{eqnarray}
 Let $V$ be a subspace of $\mathcal{E}$. From (\ref{eq:7.2}) and (\ref{eq:7.3}), we deduce that $A$ (resp. $B(z,\lambda)$) is negative definite on $V $ if and only if $C$ (resp. $D(z,\lambda))$ is negative definite on $j(V)$. Let $V^-(C)$ and $V^-(D(z,\lambda))$ be the maximal negative definite subspaces of $C$ and
 $D(z,\lambda)$ respectively.  For a critical point  $w$ of $\mathcal{A}_{H_0}$, by
 Lemma~\ref{lem:7.1} we have
 $$
 i_{\rm rel}(w)={\rm dim}\big(V^-(D(w),V^-(C))\big).
 $$
  Since $L$ is essentially self-adjoint and has compact resolvents in $L^2(M,\Sigma M)\times L^2(M,\Sigma M)$,  the spectrum of $L$ consists of eigenvalues $\bar{\lambda}_l$ satisfying

$$-\infty\downarrow\bar{\lambda}_{-k}<\cdots<\bar{\lambda}_{-1}<0<
\bar{\lambda}_1<\cdots<\bar{\lambda}_k\uparrow+\infty,\hspace{4mm}k\in \mathbb{N}.$$
Note that ${\rm Spec}(L)={\rm Spec}(D)$. (In fact, $L(u,v)=\lambda(u,v)\Longrightarrow D(u\pm v)=\lambda(u\pm v)$, and
$Du=\lambda u$ implies $L(u,v)=\lambda(u,v)$ and $u=v$).
In particular, $0$ is not an eigenvalue of $L$ by our assumption that $0\notin {\rm Spec}(D)$.
Clearly, there exist $z\in L^2(M,\Sigma M)\times L^2(M,\Sigma M)$, $\lambda, \mu\in\mathbb{R}$
such that
 $C(z,\mu)^t=\lambda(z,\mu)^t$
 if and only if
$\mu=\lambda\mu$ and $Lz=\lambda z$.

For each $l\in \mathbb{Z}\setminus\{0\}$, denote by $z_{l,1},\ldots,z_{l,m_l}$ the orthogonal  eigenvectors of $L$ with $L^2$-norm $\sqrt{2}$ corresponding to $\bar{\lambda}_l$ with multiplicity $m_l$. Given $\bar{\lambda}_k$, $k\in \mathbb{Z}$, $n\in\{1,\ldots, m_k\}$, let
$$
z=\sum\limits_{l\in\mathbb{Z}\setminus\{0\}}a_{l,1}z_{l,1}+\cdots+a_{l,m_l}z_{l,m_l}
$$
with $a_{l,1},\ldots,a_{l,m_l}\in \mathbb{C}$,
such that $D(z_{k,n},\bar{\lambda}_k)(z,\mu)^t=\lambda(z,\mu)^t$.
Then we have
\begin{equation}\label{eq:7.4}
-(z_{k,n},z)_{L^2}=\lambda\mu
\end{equation}
 and
\begin{equation}\label{eq:7.5}
\begin{cases}
\begin{array}{l}
a_{l,1}(\bar{\lambda}_l-\bar{\lambda}_k)=\lambda a_{l,1}+\mu\delta^{l,1}_{k,n},\\
a_{l,2}(\bar{\lambda}_l-\bar{\lambda}_k)=\lambda a_{l,2}+\mu\delta^{l,2}_{k,n},\\
\cdots\cdots\cdots\cdots\cdots\\
a_{l,m_l}(\bar{\lambda}_l-\bar{\lambda}_k)=\lambda a_{l,m_l}+\mu\delta^{l,m_l}_{k,n}
\end{array}
\end{cases}
\end{equation}
for each $l\in\mathbb{Z}$.
It follows from (\ref{eq:7.4}) and (\ref{eq:7.5}) that either $\mu\neq0$ and so
$$
z=-\frac{\mu}{\lambda}z_{k,n}\quad\hbox{and}\quad\lambda=\pm\|z_{k,n}\|_{L^2},
$$
or $\mu=0$ and hence $z=a_{l,1}z_{l,1}+\cdots+a_{l,m_l}z_{l,m_l}$
for some $l\neq k$ and $\lambda=\bar{\lambda}_l-\bar{\lambda}_k$.

Summarizing up the above computation, we obtain that
\begin{eqnarray*}
&&V^-(C)=\bigoplus\limits_{l\in\mathbb{N}}{\rm span}_{\mathbb{C}}\{(z_{-l,1},0),\ldots,(z_{-l,m_{-l}},0)\}
\qquad\hbox{and}\\
&&V^-(D(z_{k,n},\bar{\lambda}_k))=\bigoplus\limits_{l<k}{\rm span}_{\mathbb{C}}\{(z_{l,1},0),\ldots,(z_{l,m_l},0)\}\bigoplus{\rm span}_{\mathbb{R}}\{(\frac{\sqrt{2}}{2}z_{k,n},1)\}.
\end{eqnarray*}

Let us discuss the relative index at $(z_{k,n},\bar{\lambda}_k)$ in two cases.

\noindent{\bf Case 1}. If $k<0$, then
\begin{eqnarray*}
&&V^-(D(z_{k,n},\bar{\lambda}_k))\cap\bigl(V^-(C)\bigr)^\perp=\{0\}\qquad\hbox{and}\\
&&[V^-(D(z_{k,n},\bar{\lambda}_k))]^\perp\cap V^-(C)=\bigoplus\limits_{k\leq l<0}{\rm span}_{\mathbb{C}}\{(z_{-l,1},0),\ldots,(z_{-l,m_{-l}},0)\}.
\end{eqnarray*}
It follows that
 $$
 i_{\rm rel}(z_{k,n},\bar{\lambda}_k)=-2\sum\limits_{k\leq l<0}m_l,
 $$
 which implies
$$
\nu(p_k^+)=-1-2\sum\limits_{k+1\leq l<0}m_l\qquad\hbox{and}\qquad
\nu(p_k^-)=-2\sum\limits_{k\leq l<0}m_l.
$$

\noindent{\bf Case 2}. If $k>0$, then
$$
V^-(D(z_{k,n},\bar{\lambda}_k))\cap\bigl(V^-(C)\bigr)^\perp=\bigoplus\limits_{0<l<k}{\rm span}_{\mathbb{C}}\{(z_{l,1},0),\ldots,(z_{l,m_l},0)\}\bigoplus{\rm span}_{\mathbb{R}}\{(\frac{\sqrt{2}}{2}z_{k,n},1)\}
$$
and
$$
\bigl(V^-(D(z_{k,n},\bar{\lambda}_k))\bigr)^\perp\cap V^-(C)=\{0\}.
$$
These lead to
$$
i_{\rm rel}(z_{k,n},\bar{\lambda}_k)=1+2\sum\limits_{0<l<k}m_l,
$$
and thus
$$
\nu(p_k^+)=2\sum\limits_{0<l\leq k}m_l\qquad\hbox{and}\qquad
\nu(p_k^-)=1+2\sum\limits_{0<l<k}m_l.
$$

Since in both cases it holds that
\begin{eqnarray*}
&&\nu(p_k^-)-\nu(p_{k-1}^+)=1\quad\hbox{for}\;
k\neq 0,1, \\
&&\nu(p_{\pm 1}^-)=\pm1\quad\hbox{and}\quad\nu(p_1^-)-\nu(p_{-1}^+)=2,
\end{eqnarray*}
we have one generator for each $\nu(p_k^\pm)$ in the chain
complex $BC_*(H,K,g,h)$ for generic choice of $K$ and $g$, that is to say,
\begin{eqnarray}\label{eq:7.6}
BC_*(H_0,K,g,h)=\left\{ \begin{array}{l}
\mathbb{Z}_2,\hspace{4mm} *=\pm 1,
\\[2pt]\mathbb{Z}_2,\hspace{4mm} *=2\sum_{1\leq j\leq k}m_j,\hspace{1mm}k\in\mathbb{N},
\\[2pt]\mathbb{Z}_2,\hspace{4mm} *=1+2\sum_{1\leq j\leq k}m_j,\hspace{1mm}k\in\mathbb{N},
\\[2pt]\mathbb{Z}_2,\hspace{4mm} *=-2\sum_{1\leq j\leq k}m_{-j},\hspace{1mm}k\in\mathbb{N},
\\[2pt]\mathbb{Z}_2,\hspace{4mm} *=-1-2\sum_{1\leq j\leq k}m_{-j},\hspace{1mm}k\in\mathbb{N},
\\[2pt]0, \hspace{6mm} \hbox{otherwise}.
\end{array} \right.
\end{eqnarray}
Notice that  $\nu(p_k^+)-\nu(p_k^-)\geq3$ for $m_k>1$. But
for $m_k=1$ one can choose $(h_k,g_k)$ such that there exist precisely two flow lines from $p_k^+$ to $p_k^-$ on the connected component $\sigma_k\thickapprox S^1$. In both cases, we get
$$
\partial_{\nu(p_k^+)}=0\quad\hbox{for all}\;k\in\mathbb{Z}\setminus \{0\}.
$$
 To compute homology it remains to compute $\partial_{\nu(p_k^-)}$. Since there exists no generators of index $0$, the image of $\partial_0$ must be zero. Then ${\rm ker} \partial_{-1}=\mathbb{Z}_2$ implies that
\begin{equation}\label{eq:7.7}
RHF_{-1}(H_0)=HF_{-1}(H_0)=\mathbb{Z}_2.
\end{equation}
Furthermore, by continuation we have
\begin{theorem}\label{th:7.1}
Let $H\in C^2(\Sigma M\oplus\Sigma M)$ satisfy $({\bf H}1)-({\bf H}4)$. Then
\begin{eqnarray}\label{eq:7.8}
RHF_{-1}(H)=RHF_{-1}(H_0)=\mathbb{Z}_2.
\end{eqnarray}
\end{theorem}

\subsection{Proof of Theorem~\ref{th:1.1}}\label{sec:7.2}

By the transversality results in Section~\ref{sec:6} and Theorem~\ref{th:7.1} there exist a sequence
of smooth functions $H_n:\Sigma M\oplus\Sigma M\to\mathbb{R}$ satisfying $({\bf H}1)-({\bf H}4)$ and
the following conditions\\
$\bullet$ $H-H_n\in\mathbf{G}\;\forall n$, and $\varepsilon_n:=|H-H_n|_{C^2}\to 0$ as $n\to\infty$;\\
$\bullet$ $RHF_{-1}^{[a,b]}(H_n)\ne 0\;\forall n$ for some two real numbers $a<b$.\\
Thus $\mathcal{A}_{H_n}$ has
  critical points $(z_n,\lambda_n)$ satisfying
$$
\mathcal{A}_{H_n}(z_n,\lambda_n)\in [a,b].
$$
 From the proof of Proposition~\ref{prop:3.1} we see that $\|(z_n,\lambda_n)\|_\mathcal{E}$ is bounded by some constant $C_1$ which is independent of $H_n$. Then
\begin{eqnarray}\label{eq:7.10}
|\mathcal{A}_H(z_n,\lambda_n)-\mathcal{A}_{H_n}(z_n,\lambda_n)|&\leq&
|\lambda_n|\int_M|H(x,z_n)-H_n(x,z_n)|\notag\\
&\leq& C_1\varepsilon_n
\end{eqnarray}
and
\begin{eqnarray}\label{eq:7.11}
\|\nabla \mathcal{A}_H(z_n,\lambda_n)-\nabla \mathcal{A}_{H_n}(z_n,\lambda_n)\|_\mathcal{E}&\leq&
\left\|\mathcal{D}_s\left\{\frac{\partial H}{\partial z}(x,z_n)-\frac{\partial H_n}{\partial z}(x,z_n)\right\}\right\|\notag\\
&&+|\lambda_n|\int_M|H(x,z_n)-H_n(x,z_n)|\notag\\
&\leq&C_0\varepsilon_n+C_1\varepsilon_n,
\end{eqnarray}
where in (\ref{eq:7.11}) we have use the facts: $D_s:E_s^*\to E_s$ is bounded and $L^\infty\times L^\infty\subset E_s^*$, which
imply that the first term of the right side is bounded by $C_0\|H_z-H_{nz}\|_{L^\infty}\leq C_0\varepsilon_n$. From (\ref{eq:7.10}) we may assume that
$$
(z_n,\lambda_n)\in \mathcal{A}_H^{-1}[a-K,b+K]=:\mathcal{S}\subset \mathcal{E}
$$
for some constant $K>0$. Let
\begin{eqnarray}\label{eq:7.12}
\alpha:=\inf\limits_{(z,\lambda)\in\mathcal{S}}\nabla\mathcal{A}_H(z,\lambda).
\end{eqnarray}
\textsf{We claim that $\alpha=0$}. By a contradiction, suppose $\alpha>0$. Then (\ref{eq:7.11}) yields
\begin{eqnarray*}
\|\nabla \mathcal{A}_{H_n}(z_n,\lambda_n)\|_\mathcal{E}&\geq&\|\nabla \mathcal{A}_H(z_n,\lambda_n)\|_\mathcal{E}-\|\nabla \mathcal{A}_H(z_n,\lambda_n)-\nabla \mathcal{A}_{H_n}(z_n,\lambda_n)\|_\mathcal{E}\notag\\
&\geq&\alpha-C_0\varepsilon_n-C_1\varepsilon_n>0
\end{eqnarray*}
for very large $n$. This contradiction shows $\alpha=0$.

Now by (\ref{eq:7.12}) one can choose a sequence $\{(z_n,\lambda_n)\}_{n=1}^\infty\subset \mathcal{S}\subset\mathcal{E}$ such that
\begin{eqnarray*}
\left\{ \begin{array}{l}
\mathcal{A}_{H}(z_n,\lambda_n)\to c\in[a-K,b+K],\\
\nabla \mathcal{A}_{H}(z_n,\lambda_n)\to 0,\quad n\to \infty.
\end{array} \right.
\end{eqnarray*}
That is, $\{(z_n,\lambda_n)\}_{n=1}^\infty$ is a $(PS)_c$-sequence.
It follows from Proposition~\ref{prop:3.1} that there exists a subsequence of $\{(z_n,\lambda_n)\}_{n=1}^\infty$ converging to $(z^*,\lambda^*)$ in $\mathcal{E}$, which is a critical point of $\mathcal{A}_H$. Then (\ref{eq:2.6}) implies that $z^*=(u^*,v^*)\neq0$ and $\lambda^*\neq0$.
Put
\begin{eqnarray*}
\left\{ \begin{array}{l}
u_0={\lambda^*}^{\frac{q+1}{1-pq}}u^*,\\
v_0={\lambda^*}^{\frac{p+1}{1-pq}}v^*.
\end{array} \right.
\end{eqnarray*}
Then $(u_0,v_0)$ is a nontrivial weak solution of the Dirac system (\ref{eq:1.10}) and thus is of class $C^1$ by the elliptic regularity.

\appendix
\section{ Proof of Proposition~\ref{prop:2.1}}\label{app:A}\setcounter{equation}{0}
\setcounter{equation}{0}
\noindent{\bf Step 1.}\; {\it $\mathcal{H}$ is G\^ateaux differentiable}.
Given $z=(u,v)\in E_s$, $h=(\xi,\zeta)\in E_s$ and $t\in (-1,1)\setminus\{0\}$ we have, by the mean value theorem,
\begin{eqnarray*}
&&\frac{H(x,u(x)+t\xi(x),v(x)+t\zeta(x))-H(x,u(x),v(x))}{t}\\
&=&\langle H_u(x,u(x)+\theta_1\xi(x),v(x)+t\zeta(x)),\xi(x)\rangle+
\langle H_v(x,u(x), v(x)+\theta_2\zeta(x)),\zeta(x)\rangle
\end{eqnarray*}
for some $\theta_j=\theta(t,x, z(x), h(x))\in (0,1)$, $j=1,2$. It follows from the condition (\textbf{H}2) that
\begin{eqnarray*}
&&\left|\frac{H(x,u(x)+t\xi(x),v(x)+t\zeta(x))-H(x,u(x),v(x))}{t}\right|\\
&\le& c_1\left(1+|u(x)+\theta_1\xi(x)|^p+ |v(x)+t\zeta(x)|^{\frac{p(q+1)}{p+1}}\right)|\xi(x)|\\
&& + c_1\left(1+|u(x)|^{\frac{q(p+1)}{q+1}}+ |v(x)+\theta_2\zeta(x)|^q\right)|\zeta(x)|\\
&\le& c_1\left(1+ 2^p|u(x)|^p+ 2^p|\xi(x)|^p+ 2^{\frac{p(q+1)}{p+1}}|v(x)|^{\frac{p(q+1)}{p+1}}+ 2^{\frac{p(q+1)}{p+1}}|\zeta(x)|^{\frac{p(q+1)}{p+1}}\right)|\xi(x)|\\
&& + c_1\left(1+|u(x)|^{\frac{q(p+1)}{q+1}}+ 2^q|v(x)|+ 2^q|\zeta(x)|^q\right)|\zeta(x)|.
\end{eqnarray*}
Hereafter $c_1, c_2,\cdots$, denote constants only depending on $H$ and $M$.
By the H\"{o}ler inequality and the Sobolev embedding theorem it is easily checked that the last two lines of the above inequalities are integrable. Using the Lebesgue Dominated Convergence Theorem we deduce
\begin{eqnarray*}
&&\lim_{t\to 0}\frac{\mathcal{H}(x,u+t\xi,v+t\zeta)-\mathcal{H}(x,u,v)}{t}\\
&=&\int_M\lim_{t\to 0}\frac{H(x,u(x)+t\xi(x),v(x)+t\zeta)-H(x,u(x),v(x))}{t}dx\\
&=&\int_M\{\langle H_u(x,u(x),v(x)),\xi(x)\rangle +\langle H_v(x,u(x),v(x)),\zeta(x)\rangle\} dx
\end{eqnarray*}
because $H\in C^2(\Sigma M\oplus\Sigma M)$. Notice that
 ({\bf H}2) implies
\begin{eqnarray*}
&&\int_M|\langle H_u(x,u,v),\xi\rangle| dx\leq c_1\int_M\big(1+|u|^p+|v|^{\frac{p(q+1)}{p+1}}\big)|\xi|dx,\\
&&\int_M|\langle H_v(x,u,v),\zeta\rangle| dx\leq c_1\int_M\big(1+|u|^{\frac{q(p+1)}{q+1}}+ |v|^q\big)|\xi|dx.
\end{eqnarray*}
It follows from H\"{o}ler inequalities, (\ref{eq:1.4.1}) and Sobolev embeddings that
\begin{eqnarray*}
&&\int_M|\langle H_u(x,u,v),\xi\rangle|dx\leq C\big(1+\|u\|_{s,2}^p+\|v\|_{1-s,2}^{\frac{p(q+1)}{p+1}}\big)\|\xi\|_s,\\
&&\int_M|\langle H_v(x,u,v),\zeta\rangle|dx\leq C\big(1+\|u\|_{s,2}^{\frac{q(p+1)}{q+1}}+\|v\|_{1-s,2}^q\big)
\|\zeta\|_{1-s}.
\end{eqnarray*}
Hence  the G\^ateaux derivative $D\mathcal{H}(z)=D\mathcal{H}(u,v)$ exists and
\begin{eqnarray}\label{e:A.1}
D\mathcal{H}(z)h&=&D\mathcal{H}(u,v)(\xi,\zeta)\nonumber\\
&=&\int_M\{\langle H_u(x,u(x),v(x)),\xi(x)\rangle +\langle H_v(x,u(x),v(x)),\zeta(x)\rangle\}dx\nonumber\\
&=&\int_MH_z(x,z(x))h(x)dx.
\end{eqnarray}

\noindent{\bf Step 2.}\; {\it $D\mathcal{H}:E_s\to E_s^\ast$ is continuous and thus $\mathcal{H}$
has continuous (Fr\'echet) derivative $\mathcal{H}'=D\mathcal{H}$. Consequently, the gradient $\nabla\mathcal{H}(z)$ at $z$ is given by
\begin{equation}\label{e:A.2}
\nabla\mathcal{H}(z)=\mathcal{D}_sH_z(\cdot,z).
\end{equation}}
Let
$$\hat{r}_1=\frac{2n}{n-2s},\quad\hat{r}_2=\frac{2n}{n-2(1-s)}.
$$
({\textbf H2}) implies that if $r_1\geq p+1$ and $r_2\geq q+1$ then there exists a constant $c_3>0$ such that
\begin{eqnarray}\label{e:A.3}
|H_u(x,u,v)|\leq c_3(1+|u|^{r_1-1}+|v|^{r_2(r_1-1)/r_1}),\label{A3:1}\\
|H_v(x,u,v)|\leq c_3(1+|u|^{r_1(r_2-1)/r_2}+|v|^{r_2-1}).\label{A3:2}
\end{eqnarray}
Obviously, the above two inqualities yield
\begin{eqnarray}\label{e:A.4}
|H_u(x,u,v)|^{r_1/(r_1-1)}\leq c_4(1+|u|^{r_1}+|v|^{r_2}),\label{A4:1}\\
|H_v(x,u,v)|^{r_2/(r_2-1)}\leq c_4(1+|u|^{r_1}+|v|^{r_2}).\label{A4:2}
\end{eqnarray}
By the Sobolev embedding we can find two constants $C_i$, $i=1,2$, such that
\begin{eqnarray*}
 &&\|u\|_{L^{2n/(n-2s)}}\leq C_1\|u\|_s\quad\forall u\in H^s(M,\Sigma M),\\
 &&\|v\|_{L^{2n/(n-2(1-s))}}\leq C_2\|v\|_{1-s}\quad\forall v\in H^{1-s}(M,\Sigma M).
\end{eqnarray*}
In particular, there is a constant $a_{r_1,r_2}>0$ such that for any $z=(u,v)\in E_s=H^s(M,\Sigma M)\times H^{1-s}(M,\Sigma M)$ we have
\begin{equation}\label{e:A.4+}
\|u\|_{L^{r_1}}+\|v\|_{L^{r_2}}\leq a_{r_1,r_2}\|z\|.
\end{equation}
Given $z=(u,v),h=(h_1,h_2)\in E_s$, combining (\ref{A3:1}) and (\ref{e:A.4}) gives
\begin{eqnarray}\label{e:A.5}
&&\int_M|H_u(x,z(x)+h(x))-H_u(x,z(x))|^{r_1/(r_1-1)}dx\nonumber
\\&\leq&c_5\int_M(1+|u|^{r_1}+|h_1|^{r_1}+|v|^{r_2}+|h_2|^{r_2}
)dx\nonumber \\
&\leq&c_5(1+\|z\|^{r_1}+\|z\|^{r_2}+\|h\|^{r_1}+\|h\|^{r_2}).
\end{eqnarray}
Similarly, from (\ref{A3:2}) and (\ref{e:A.4}) we arrive at
\begin{eqnarray}\label{e:A.6}
&&\int_M|H_v(x,z(x)+h(x))-H_v(x,z(x))|^{r_2/(r_2-1)}dx\nonumber
\\
&\leq&c_6(1+\|z\|^{r_1}+\|z\|^{r_2}+\|h\|^{r_1}+\|h\|^{r_2}).
\end{eqnarray}
By the definition and the H\"{o}lder inequality we get
\begin{eqnarray}\label{e:A.7}
&&\|D\mathcal{H}(z+h)- D\mathcal{H}(z)\|_{E_s^\ast}=
\sup_{\|g\|\le 1}\left| D\mathcal{H}(z+h)- D\mathcal{H}(z), g\rangle\right|\nonumber\\
&=&\sup_{\|g\|\le 1}\bigg[\int_M\big(|H_u(x,z(x)+h(x))-H_u(x,z(x))||g_1(x)|\nonumber\\
&&+|H_v(x,z(x)+h(x))-H_v(x,z(x))||g_2(x)|\big)dx\bigg]\nonumber\\
&\leq&\sup_{\|g\|\le 1}\bigg[\|g_1\|_{L^{r_1}}\times \left(\int_M|H_u(x,z(x)+h(x))-H_u(x,z(x))|^{r_1/(r_1-1)}dx\right)^{(r_1-1)/r_1}\nonumber\\
&&+\|g_2\|_{L^{r_2}}\times\left(\int_M|H_u(x,z(x)+h(x))-H_u(x,z(x))|^{r_2/(r_2-1)}dx\right)
^{(r_2-1)/r_2}\bigg]\nonumber\\
&\leq&a_{r_1,r_2}\bigg[ \left(\int_M|H_u(x,z(x)+h(x))-H_u(x,z(x))|^{r_1/(r_1-1)}dx\right)^{(r_1-1)/r_1}\nonumber\\
&&+\left(\int_M|H_v(x,z(x)+h(x))-H_v(x,z(x))|^{r_2/(r_2-1)}dx\right)
^{(r_2-1)/r_2}\bigg],
\end{eqnarray}
where we have used (\ref{e:A.4}) in the last inequality. We deduce from (\ref{A4:1}) and (\ref{e:A.5}) that the Nemytski map
$$
N_{H_u}: L^{r_1}(M, \Sigma M)\times L^{r_2}(M, \Sigma M)\to L^{\frac{r_1}{r_1-1}}(M, \Sigma M),\;
h\mapsto H_u(\cdot,h(\cdot))
$$
is continuous. Similarly, (\ref{A4:2}) and (\ref{e:A.6}) imply the Nemytski map
$$
N_{H_v}: L^{r_1}(M, \Sigma M)\times L^{r_2}(M, \Sigma M)\to L^{\frac{r_2}{r_2-1}}(M, \Sigma M),\;
h\mapsto H_v(\cdot,h(\cdot))
$$
is also continuous. Then by the Sobolev embedding we have
\begin{eqnarray}\label{e:A.8}
\int_M|H_u(x,z(x)+h(x))-H_u(x,z(x))|^{r_1/(r_1-1)}dx\to 0,\label{A8:1}\\
\int_M|H_v(x,z(x)+h(x))-H_v(x,z(x))|^{r_2/(r_2-1)}dx\to 0\label{A8:2}
\end{eqnarray}
as $\|h\|\to 0$. Combining (\ref{e:A.7}), (\ref{A8:1}) and (\ref{A8:2}) yields
$$
\|D\mathcal{H}(z+h)-D\mathcal{H}(z)\|_{\mathcal{L}(E_s, E_s^\ast)}\to 0\quad \hbox{as}\; \|h\|\to 0.
$$
\noindent{\bf Step 3.}\; {\it $\mathcal{H}'$ is a compact map}. Suppose that $(z_k)\subset E_s$ is bounded. Passing to a subsequence, one may assume that $z_k$ converges weakly in $E_s$ to $z=(u,v)$. By the Rellich embedding theorem and the continuousness of $\mathcal{H}'=D\mathcal{H}$ we obtain that
$\mathcal{H}^\prime(z_k)\to\mathcal{H}^\prime(z)$ as $k\to \infty$.\\
\noindent{\bf Step 4.}\; {\it $\mathcal{H}'$ is of class $C^1$}. Given $x\in M$, $z=(u,v),h=(h_1,h_2),g=(g_1,g_2)\in E_s$, and $t\in(-1,1)\backslash\{0\}$, the mean value theorem implies that
\begin{eqnarray}\label{e:A.9}
&&\mathcal{H}'(z+th)g-\mathcal{H}'(z)g\nonumber\\
&=&\int_M(H_z(x,z(x)+th(x))-H_z(x,z(x)))g(x)dx\nonumber\\
&=&\int_M\langle H_{zz}(x,z(x)+\theta th(x))th(x),g(x)\rangle dx
\end{eqnarray}
with $\theta=\theta(x,t,z(x),h(x))\in(0,1)$. Let $s_1:=\frac{n}{2s}$ and $s_2:=\frac{n}{2(1-s)}$. Then we have
{\setlength\abovedisplayskip{1pt}
\setlength\belowdisplayskip{1pt}
\begin{eqnarray}\label{e:A.11}
\frac{2}{\hat{r}_1}+\frac{1}{s_1}=1\label{holder:1}\quad\hbox{and}\quad
\frac{2}{\hat{r}_2}+\frac{1}{s_2}=1\label{holder:2}.
\end{eqnarray}
(\textbf{H3}) implies that there is a constant $a_1>0$ such that for any $z=(u,v)\in\Sigma_xM\oplus\Sigma_xM$ it holds that
{\setlength\abovedisplayskip{1pt}
\setlength\belowdisplayskip{1pt}
\begin{eqnarray}\label{Nemy:1}
&&|H_{uu}(x,z(x))|^{s_1}\leq a_1\big(1+|u(x)|^{p-1}\big)^{s_1},\nonumber\\
&&|H_{vv}(x,z(x))|^{s_2}\leq a_1\big(1+|v(x)|^{q-1}\big)^{s_2},\nonumber\\
&&|H_{uv}(x,z(x))|^{n}\leq a_1,\quad |H_{vu}(x,z(x))|^{n}\leq a_1.
\end{eqnarray}
Notice that $0<(p-1)s_1< \hat{r}_1$ and $0<(q-1)s_2< \hat{r}_2$, using the H\"{o}lder inequality gives
{\setlength\abovedisplayskip{1pt}
\setlength\belowdisplayskip{1pt}
\begin{eqnarray}\label{e:A.12}
&&\int_M |H_{uu}(x,z(x))+\theta th(x))-H_{uu}(x,z(x))|^{s_1}dx\nonumber\\
&\leq&a_2\int_M \big(1+|u(x)|^{s_1(p-1)}+|h_1(x)|^{s_1(p-1)}\big)dx\nonumber\\
&\leq&a_3\big(1+\|u\|_{L^{\hat{r}_1}}^{s_1(p-1)}+\|h_1\|_{L^{\hat{r}_1}}^{s_1(p-1)}\big),
\end{eqnarray}
{\setlength\abovedisplayskip{1pt}
\setlength\belowdisplayskip{1pt}
\begin{eqnarray}\label{e:A.13}
&&\int_M |H_{vv}(x,z(x))+\theta th(x))-H_{vv}(x,z(x))|^{s_2}dx\nonumber\\
&\leq&a_4\int_M \big(1+|v(x)|^{s_2(q-1)}+|h_2(x)|^{s_2(q-1)}\big)dx\nonumber\\
&\leq&a_5\big(1+\|v\|_{L^{\hat{r}_2}}^{s_2(q-1)}+\|h_2\|_{L^{\hat{r}_2}}^{s_2(q-1)}\big),
\end{eqnarray}
{\setlength\abovedisplayskip{1pt}
\setlength\belowdisplayskip{1pt}
\begin{eqnarray}\label{e:A.14}
\int_M |H_{uv}(x,z(x))+\theta th(x))-H_{uv}(x,z(x))|^{n}dx\leq a_6
\end{eqnarray}
and
{\setlength\abovedisplayskip{1pt}
\setlength\belowdisplayskip{1pt}
\begin{eqnarray}\label{e:A.15}
\int_M |H_{vu}(x,z(x))+\theta th(x))-H_{vu}(x,z(x))|^{n}dx\leq a_6
\end{eqnarray}
{\setlength\abovedisplayskip{1pt}
\setlength\belowdisplayskip{1pt}
It follows from (\ref{e:A.9}) and the H\"{o}lder inequality that
\begin{eqnarray*}
&&\left\|\frac{1}{t}(\mathcal{H}'(z+th)-\mathcal{H}'(z))-H_{zz}(\cdot,z)h\right\|_{E^\ast_s}\\
&\le&\sup_{\|g\|\le 1}
\left|\frac{1}{t}(\mathcal{H}'(z+th)g-\mathcal{H}'(z)g)-\langle H_{zz}(\cdot,z)h,g\rangle\right|\\
&\le&\sup_{\|g\|\le 1}\left|\int_M\langle H_{zz}(x,z(x)+\theta th(x))h(x)-H_{zz}(x,z(x))h(x),g(x)\rangle dx\right|\nonumber\\
&\leq& \sup_{\|g\|\le 1}\int_M\big[|H_{uu}(x,z(x)+\theta th(x))-H_{uu}(x,z(x))||h_1(x)||g_1(x)|\nonumber\\
&&+|H_{uv}(x,z(x)+\theta th(x))-H_{uv}(x,z(x))||h_2(x)||g_1(x)|\nonumber\\
&&+|H_{vu}(x,z(x)+\theta th(x))-H_{uv}(x,z(x))||h_1(x)||g_2(x)|\nonumber\\
&&+|H_{vv}(x,z(x)+\theta th(x))-H_{uv}(x,z(x))||h_2(x)||g_2(x)|\big]dx
\nonumber\\
&\leq& \sup_{\|g\|\le 1}\big[\|H_{uu}(\cdot,z+\theta th)-H_{uu}(\cdot,z)\|_{L^{s_1}}\|h_1\|_{L^{\hat{r}_1}}\|g_1\|_{L^{\hat{r}_1}}\nonumber\\
&&+\|H_{uv}(\cdot,z+\theta th)-H_{uv}(\cdot,z)\|_{L^n}\|h_2\|_{L^{\hat{r}_2}}\|g_1\|_{L^{\hat{r}_1}}\nonumber\\
&&+\|H_{vu}(\cdot,z+\theta th)-H_{vu}(\cdot,z)\|_{L^n}\|h_1\|_{L^{\hat{r}_1}}\|g_2\|_{L^{\hat{r}_2}}\nonumber\\
&&+\|H_{vv}(\cdot,z+\theta th)-H_{vv}(\cdot,z)\|_{L^{s_2}}\|h_2\|_{L^{\hat{r}_2}}\|g_2\|_{L^{\hat{r}_2}}\big].
\end{eqnarray*}
Combining (\ref{e:A.12})-(\ref{e:A.15}), the Sobolev inequality and the Lebesgue dominated convergence theorem implies that the right hand side of the last one of the above inequalities tends to $0$ as $t\to 0$. Therefore $\mathcal{H}':E_s\to E_s^\ast$ is G\^atuaux differentiable and
\begin{eqnarray}\label{e:A.16}
D\mathcal{H}'(z)h=H_{zz}(\cdot,z)h.
\end{eqnarray}
{\setlength\abovedisplayskip{1pt}
\setlength\belowdisplayskip{1pt}
For $z,h\in E_s$ it holds that
\begin{eqnarray}\label{e:A.17}
&&\|D\mathcal{H}'(z+h)-D\mathcal{H}'(z)\|_{\mathcal{L}(E_s, E_s^\ast)}=\sup_{\|g\|\le 1}
\|D\mathcal{H}'(z+h)g-D\mathcal{H}'(z)g\|_{E_s^\ast}\nonumber\\
&\le& \sup_{\|g\|\le 1}
\sup_{\|\kappa\|\le 1}|\langle D\mathcal{H}'(z+h)g-D\mathcal{H}'(z)g,\kappa\rangle|\nonumber\\
&=&\sup_{\|g\|\le 1}
\sup_{\|\kappa\|\le 1}\left|\int_M\langle H_{zz}(x,z(x)+h(x))g(x)-H_{zz}(x,z(x))g(x),\kappa(x)\rangle dx\right|\nonumber\\
&\leq& \sup_{\|g\|\le 1}\sup_{\|\kappa\|\le 1}\int_M\big[|H_{uu}(x,z(x))+h(x))-H_{uu}(x,z(x))||g_1(x)||\kappa_1(x)|\nonumber\\
&&+|H_{uv}(x,z(x))+h(x))-H_{uv}(x,z(x))||g_2(x)||\kappa_1(x)|\nonumber\\
&&+|H_{vu}(x,z(x))+h(x))-H_{uv}(x,z(x))||g_1(x)||\kappa_2(x)|\nonumber\\
&&+|H_{vv}(x,z(x))+h(x))-H_{uv}(x,z(x))||g_2(x)||\kappa_2(x)|\big]dx
\nonumber\\
&\leq&\sup_{\|g\|\le 1}\sup_{\|\kappa\|\le 1}
\big[\|H_{uu}(\cdot,z+h)-H_{uu}(\cdot,z)\|_{L^{s_1}}\|g_1\|_{L^{\hat{r}_1}}\|\kappa_1\|_{L^{\hat{r}_1}}\nonumber\\
&&+\|H_{uv}(\cdot,z+h)-H_{uv}(\cdot,z)\|_{L^n}\|g_2\|_{L^{\hat{r}_2}}\|\kappa_1\|_{L^{\hat{r}_1}}\nonumber\\
&&+\|H_{vu}(\cdot,z+h)-H_{vu}(\cdot,z)\|_{L^n}\|g_1\|_{L^{\hat{r}_1}}\|\kappa_2\|_{L^{\hat{r}_2}}\nonumber\\
&&+\|H_{vv}(\cdot,z+h)-H_{uv}(\cdot,z)\|_{L^{s_2}}\|g_2\|_{L^{\hat{r}_2}}\|\kappa_2\|_{L^{\hat{r}_2}}\big]\nonumber\\
&\leq&
a_{\hat{r}_1,\hat{r}_2}^2\big[\|H_{uu}(\cdot,z+h)-H_{uu}(\cdot,z)\|_{L^{s_1}}+\|H_{uv}(\cdot,z+h)-H_{uv}(\cdot,z)\|_{L^n}
\nonumber\\
&&+\|H_{vu}(\cdot,z+h)-H_{vu}(\cdot,z)\|_{L^n}+\|H_{vv}(\cdot,z+h)-H_{uv}(\cdot,z)\|_{L^{s_2}}\big],
\end{eqnarray}
where we have used (\ref{e:A.4}) in the last inequality. Obviouly, (\ref{Nemy:1}) implies that there is a constant $b_1>0$ such that
\begin{eqnarray}\label{Nemy:2}
&&|H_{uu}(x,z(x))|^{s_1}\leq b_1\big(1+|u(x)|^{\hat{r}_1}\big),\quad
|H_{vv}(x,z(x))|^{s_2}\leq b_1\big(1+|v(x)|^{\hat{r}_2}\big),\nonumber\\
&&|H_{uv}(x,z(x))|^{n}\leq b_1,\quad |H_{vu}(x,z(x))|^{n}\leq b_1.
\end{eqnarray}
Then the above inequalities imply that the Nemytski maps
\begin{eqnarray}
&&N_{H_{uu}}: L^{\hat{r}_1}(M,\Sigma M)\times L^{\hat{r}_2}(M,\Sigma M)\to L^{s_1}(M, {\rm End}(\Sigma M)),\;
h\mapsto H_{uu}(\cdot,h(\cdot)),\nonumber\\
&&N_{H_{vv}}: L^{\hat{r}_1}(M,\Sigma M)\times L^{\hat{r}_2}(M,\Sigma M)\to L^{s_2}(M, {\rm End}(\Sigma M)),\;
h\mapsto H_{vv}(\cdot,h(\cdot)),\nonumber\\
&&N_{H_{uv}}: L^{\hat{r}_1}(M,\Sigma M)\times L^{\hat{r}_2}(M,\Sigma M)\to L^{n}(M, {\rm End}(\Sigma M)),\;
h\mapsto H_{uv}(\cdot,h(\cdot))\quad\hbox{and}\nonumber\\
&&N_{H_{vu}}: L^{\hat{r}_1}(M,\Sigma M)\times L^{\hat{r}_2}(M,\Sigma M)\to L^{n}(M, {\rm End}(\Sigma M)),\;
h\mapsto H_{vu}(\cdot,h(\cdot))\nonumber
\end{eqnarray}
are all continuous. From this, (\ref{e:A.17}) and the Sobolev embedding we deduce that
$$
\|D\mathcal{H}'(z+h)-D\mathcal{H}'(z)\|_{\mathcal{L}(E_s, E_s^\ast)}\to 0
$$
as $\|h\|\to 0$. The desired result is proved.
\qed


{\bf Acknowledgement}. The authors want to thank Dr. Maalaoui for the fruitful discussions and the useful suggestions on the grading of Rabinowitz-Floer homology.




\begin{thebibliography}{SK}

\bibitem{Ada} R. Adams, \emph{Sobolev Space}. Academic Press, New York (1975).


\bibitem{Abb} A. Abbondandolo, A new cohomology for the Morse theory of strongly indefinite functionals on Hilbert spaces. {\it Topol. Methods Nonlinear Anal.} {\bf 9} (1997), no. 2, 325 - 382.

\bibitem{AbM} A. Abbondandolo and P. Majer, Morse homology on Hilbert spaces. {\it Comm. Pure Appl. Math.} {\bf 54} (2001), no. 6, 689 - 760.

\bibitem{AbS} A. Abbondandolo and M. Schwarz, Estimates and computations in Rabinowitz-Floer homology. {\it J. Topol. Anal.} {\bf 1} (2009), 307 - 405.

\bibitem{AlF} P. Albers and U. Frauenfelder, Spectral invariants in Rabinowitz-Floer homology and global Hamiltonian perturbations. {\it J. Mod. Dyn.} {\bf 4} (2010), 329 - 357.

\bibitem{Amm} B. Ammann, \emph{A variational Problem in Conformal Spin Geometry}. Habilitationsschift, Universit\"{a}t Hamburg 2003.

\bibitem{AnV} S. Angenent and R. van der Vorst, A superquadratic indefinite elliptic system and its Morse-Conley-Floer homology. {\it Math. Z.} {\bf 231} (1999), no. 2, 203 - 248.

\bibitem{BaL} A. Bahri and P.L. Lions, Solutions of superlinear elliptic equations and their Morse indices. {\it Commun. Pure Appl. Math.} {\bf 45}  (1992), 1205 - 1215.

\bibitem{BaD1} T. Bartsch and Y. Ding, Homoclinic solutions of an
    infinite-dimensional Hamiltonian system. {\it Math. Z.} {\bf 240} (2002), 289 - 310.

\bibitem{BaD2} T. Bartsch and Y. Ding, Periodic solutions of superlinear beam and membrane equations with perturbations from symmetry. {\it Nonlinear Analysis} {\bf 44} (2001), 727 - 748.

\bibitem{BeR} V. Benci and P.H. Rabinowitz, Periodic solutions of Hamiltonian systems. {\it Commun. Pure Appl. Math.} {\bf 31} (1978), 157 - 184.

\bibitem{Bou} F. Bourgeois, \emph{A Morse-Bott approach to contact homology}. Symplectic and contact topology: interactions and perspectives (Toronto, ON/Montreal, QC, 2001), 55 - 77, Fields Inst. Commun., 35, Amer. Math. Soc., Providence, RI, 2003.

\bibitem{BoO} F. Bourgeois and A. Oancea, Symplectic homology, autonomous Hamiltonians, and Morse-Bott moduli spaces. {\it Duke Math. J.} {\bf 146} (2009), no. 1, 71 - 174.

\bibitem {CJL} Q. Chen, J. Jost,  J. Li, and G. Wang, Dirac-harmonic maps. {\it Math. Z.} {\bf 254} (2006), 409 - 432.

\bibitem {CJW1} Q. Chen, J. Jost and G. Wang, Nonlinear Dirac equations on Riemann surfaces. {\it Ann. Global Anal. Geom.} {\bf 33} (2008), 253 - 270.
    
\bibitem {CJW2} Q. Chen, J. Jost and G. Wang, The maximum principle and the Dirichlet problem for Dirac-harmonic maps. {\it Calc. Var. Partial Differential Equations} {\bf 47} (2013), no. 1-2, 87 - 116.

\bibitem {CiF} K. Cieliebak and U. Frauenfelder, A Floer homology for exact contact embeddings. {\it Pacific J. Math.} {\bf 239} (2009), no. 2, 251 - 316.

\bibitem {CFA} K. Cieliebak, U. Frauenfelder and A. Oancea, Rabinowitz Floer homology and symplectic homology. {\it Ann. Sci. \'{E}c. Norm. Sup¨¦r.} {\bf 43} (2010), no. 6, 957 - 1015.

\bibitem{Fed} P. Felmer and D.G. deFigueiredo, On superquadratic elliptic systems. {\it Trans. Amer. Math. Soc.} {\bf 343} (1994), 99 - 116.


\bibitem{Fra} U. Frauenfelder, The Arnold-Givental conjecture and moment Floer homology. {\it Int. Math. Res. Not.} 2004, no. 42, 2179 - 2269.


\bibitem{Fri} T. Friedrich, On the spinor representation of surfaces in Euclidean 3-space. {\it J. Geom. Phy.} {\bf 28} (1998), 143¨C157.

\bibitem{Gin} N. Ginoux, \emph{The Dirac Spectrum}, Lecture Notes in Math., vol. 1976, Springer, Dordrecht-heidelberg-London-New York, 2009.

\bibitem{GoL} W. Gong and G. Lu, On Dirac equation with a potential and critical Sobolev exponent. {\it Commun. Pure Appl. Anal.} {\bf 14} (2015), 2231 - 2263.

\bibitem{Fri} T. Friedrich, \emph{Dirac Operators in Riemannian Geometry}, Grad. Stud. Math., vol. 25, Amer. Math. Soc.,Providence, RI, 2000.


\bibitem{HiP}  A. Hinrichs and A. Pietsch,  $p$-nuclear operators in the sense of Grothendieck. {\it Math. Nachr.} {\bf 283} (2010), no. 2, 232 - 261.

\bibitem{Hom} L. H\"{o}mander, \emph{The analysis of linear partial differential operators}. I. Distribution theory and Fourier analysis. Grundlehren der Mathematischen Wissenschaften, 256. Springer, Berlin-New York, 1983.

\bibitem{HuV} J. Hulshof and R. van der Vorst, Differential systems with strongly indefinite variational structure, {\it J. Funct. Anal.} {\bf 114} (1993), 32-58.

\bibitem{Iso1} T. Isobe, Existence results for solutions to nonlinear Dirac equations on compact spin manifolds. {\it Manuscripta math}. {\bf 135} (2011), 329 - 360.

\bibitem{Iso2} T. Isobe, Nonlinear Dirac equations with critical nonlinearities on compact spin manifolds. {\it J.Funct. Anal.} {\bf 260} (2011), 253 - 307.

\bibitem{Iso3} T. Isobe, A perturbation method for spinorial Yamabe type equations on $S^m$ and its application. {\it Math. Ann.} {\bf 355} (2013), 1255-1299.

\bibitem{Kat} T. Kato. \emph{Perturbation Theory for Linear Operators}. Springer, Grundlehren Math. Wiss. 132, 1980.

\bibitem{KrP} A. Kriegl and P. W. Michor, \emph{The convenient setting of global analysis}. Mathematical Surveys and Monographs, 53. American Mathematical Society, Providence, RI, 1997.

\bibitem{KrS} W. Kryszewski and A. Szulkin, An infinite-dimensional Morse theory with applications. {\it Trans. Amer. Math. Soc.} {\bf 349} (1997), 3181 - 3234.

\bibitem{LaM} H.B. Lawson and M.L. Michelson, \emph{Spin Geometry}. Princeton University Press, 1989.

\bibitem{Maa} A. Maalaoui, Rabinowitz-Floer homology for superquadratic Dirac equations on spin manifolds. {\it  J. Fixed Point Theory Appl.} {\bf 13} (2013), 175-199.

\bibitem{MaM} A. Maalaoui and V. Martino, The Rabinowitz-Floer homology for a class of semilinear problems and applications. {\it J. Funct. Anal.} {\bf 269} (2015), 4006-4037.

\bibitem{Rab} P. Rabinowitz, Periodic solutions of Hamiltonian systems. {\it Comm. Pure. Appl. Math.} {\bf 31} (1978), no 2, 157-184.

\bibitem{Rab1} P. Rabinowitz, \emph{Minimax methods in critical point theory with applications to differential equations}. CBMS Regional Conference Series in Mathematics, 65. \emph{Published for the Conference Board of the Mathematical Sciences, Washington, DC; by the American Mathematical Society, Providence, RI}, 1986.

\bibitem{SaD}  D. Salamon and E. Zehnder, Morse theory for periodic solutions of Hamiltonian systems and the Maslov index. {\it Comm. Pure Appl. Math.} {\bf 45} (1992), no. 10, 1303 - 1360.

\bibitem{ScM} M. Schwarz, \emph{Morse homology}. Progress in Mathematics, 111. Birkh\"{a}user Verlag, Basel, 1993.

\bibitem{Sch} T. Schwartz, \emph{Nonlinear Functional Analysis}. Gordon and Breach, New York, 1969.

\bibitem{Sma} S. Smale,  An infinite dimensional version of Sard's theorem. {\it Amer. J. Math.} {\bf 87} (1965) 861 - 866.

\bibitem{YJLu} Xu Yang, Rongrong Jin, Guangcun Lu, Solutions of Dirac equations on compact spin manifolds via saddle point reduction, {\it preprint}, May 2016.
\end{thebibliography}
\end{document}